\documentclass[10pt]{amsart}
\usepackage{amsmath}
\usepackage{amsthm}
\usepackage{amscd}
\usepackage{amssymb}
\usepackage{amsfonts}
\usepackage{graphics}
\usepackage[dvipsnames]{xcolor}
\usepackage[latin1]{inputenc}
\usepackage[T1]{fontenc}
\usepackage{soul}
\usepackage{cancel}
\usepackage{geometry}
\usepackage{cite}
\usepackage{enumitem}

\date{}

\newlength{\defbaselineskip}
\long\def\salta#1{\relax}

\theoremstyle{plain}
\newtheorem{theorem}{Theorem}[section]
\newtheorem{proposition}[theorem]{Proposition}
\newtheorem{lemma}[theorem]{Lemma}

\theoremstyle{definition}
\newtheorem{definition}[theorem]{Definition}

\newtheorem{remark}[theorem]{Remark}

\theoremstyle{remark}

\definecolor{jose}{rgb}{0.3,0.4,0.8}

\def\dys{\displaystyle}

\def\t1p0{T^{1,p}_{0}(\Omega)}

\def\m2{M^{\frac{N(p-1)}{N-1}}(\Omega)}

\def\div{\text{{\rm div}}}

\def\w-1p'{W^{-1,p'}(\Omega)}
\def\pw-1p'u{L^{p'}(0,1;W^{-1,p'}(\Omega))}

\def\dys{\displaystyle}

\def\lp'n{(L^{p'}(\Omega))^{N}}

\numberwithin{equation}{section}

\title[Blow-up for singular problems with natural growth]{A blow-up approach for singular elliptic problems with natural growth in the gradient}

\keywords{Nonlinear elliptic equations, Singular gradient terms, Blow-up argument.
\\
\indent Research supported by PGC2018-096422-B-I00 (MCIU/AEI/FEDER, UE), Junta de Andaluc\'ia FQM-116 and Programa de Contratos Predoctorales del Plan Propio de la Universidad de Granada.
}

\begin{document}

\maketitle

\begin{center}
{\small SALVADOR L\'OPEZ-MART\'INEZ\\[0.4 cm] 
Inria, Univ. Lille, CNRS, UMR 8524 - Laboratoire Paul Painlev\'e, F-59000 
\\
Lille, France
\\[0.4 cm]
\textsl{Email adress: salvador.lopez-martinez@inria.fr}}
\end{center}

\hrule

\begin{abstract}
We prove existence and nonexistence results concerning elliptic problems whose basic model is
\begin{equation*}
\begin{cases}
\displaystyle-\Delta u+\mu(x)\frac{|\nabla u|^2}{(u+\delta)^\gamma}= \lambda u^p, &x\in \Omega,
\\
u> 0, &x\in \Omega,
\\
u=0, &x\in\partial\Omega,
\end{cases}
\end{equation*}
where $\Omega\subset\mathbb{R}^N (N\geq 3)$ is a bounded smooth domain, $\lambda>0$, $p>1$, $\delta\geq 0$, $\gamma>0$ and $\mu\in L^\infty(\Omega)$. The main achievement resides in handling a possibly singular ($\delta=0$) first order term having a nonconstant coefficient $\mu$ in the presence of a superlinear zero order term. Our approach for the existence results is based on fixed point theory. With the aim of applying it, a previous analysis on a related non-homogeneous problem is carried out. The required a priori estimates are proven via a blow-up method.
\end{abstract}

\hrule



\section{Introduction}

Let $\Omega\subset\mathbb{R}^N (N\geq 3)$ be a bounded domain of class $\mathcal{C}^{2}$, $g:\Omega\times (0,+\infty)\to\mathbb{R}$ be a Carath\'eodory function, and $f:[0,+\infty)\to [0,+\infty)$ be a continuous function. In this work we will study the existence of solution to elliptic problems of the following form:
\begin{equation}\label{lambdaproblem}
\begin{cases}
\dys-\Delta u+g(x,u)|\nabla u|^2= \lambda f(u), &x\in \Omega,
\\
u> 0, &x\in \Omega,
\\
u=0, &x\in\partial\Omega,
\end{cases}
\tag{$P_\lambda$}
\end{equation}
where $\lambda>0$ is a parameter. The precise conditions on functions $g,f$ and the statements of the main results regarding problem \eqref{lambdaproblem} will be shown in Section~\ref{mainresults}. For the sake of a clear presentation, we consider for now a model problem:
\begin{equation}\label{model1}
\begin{cases}
\dys-\Delta u+\mu(x)\frac{|\nabla u|^2}{(u+\delta)^\gamma}= \lambda u^p, &x\in \Omega,
\\
u> 0, &x\in \Omega,
\\
u=0, &x\in\partial\Omega,
\end{cases}
\end{equation}
where $\delta\geq 0$, $\gamma>0$, $p>1$ and $\mu\in L^\infty(\Omega)$. Our main goal is to allow $\mu$ to be nonconstant and even sign-changing, paying special attention to the singular case $\delta=0$. 


Problem \eqref{model1} for $\mu\equiv 0$ becomes semilinear; as $p>1$, it is usually referred to as superlinear. This superlinear case is classical and has been extensively studied in the literature. Indeed, both variational (in \cite{AmbrosettiRabinowitz}) and topological (in \cite{deFigLiNuss, GidasSpruck}) methods can be used to prove the existence of a solution to \eqref{model1} for all $\lambda>0$ provided $\mu\equiv 0$ and $p\in (1,2^*-1)$, where $2^*=\frac{2N}{N-2}$. It is well-known that Pohozaev's identity (see \cite{Pohozaev}) implies that the restriction $p<2^*-1$ is  necessary for the existence of solution to the superlinear problem if the domain is starshaped. 

The study of problem \eqref{model1} for a nontrivial $\mu$ was initiated in \cite{OrsinaPuel}. There, the authors consider $p\in (1,2^*-1)$, $\delta>0$ and $\mu\equiv \text{constant}>0$. In this setting, they prove several results which can be divided into two classes: on the one hand, those which lead to a set of solutions similar to that for the classical semilinear problem (i.e. existence for all $\lambda>0$) and, on the other hand, those which present differences such as nonexistence of solution for $\lambda>0$ small. In fact, the semilinear-like behavior is achieved provided $\gamma>1$, while the differences appear if either $\gamma\in (0,1)$ or $\gamma=1$ and $\mu>p$. In several proofs, the authors of \cite{OrsinaPuel} make use of the change of variable
\begin{equation}\label{changeofvariables}
v=\psi(u)=\int_0^u e^{-\int_0^s \frac{\mu}{(t+\delta)^\gamma}dt}ds.
\end{equation}
It is easy to check formally that, if $\mu\equiv\text{constant}$ and $\delta>0$, then the transformation \eqref{changeofvariables} turns \eqref{model1} into a semilinear problem in the variable $v$. Thus, roughly speaking, the gradient term is removed and semilinear techniques (such as variational methods) can be applied in general. However, such a transformation can be performed only if $\mu$ is constant.

We remark that problem \eqref{model1} in the  case $\gamma=1$ is specially interesting because of the condition $\mu>p$ that appears in \cite{OrsinaPuel}, which shows that the interaction between the gradient term and the superlinear term in \eqref{model1} plays an important role. To this respect, some results concerning the case $\gamma=1$ (always with $\delta>0$ and $\mu\equiv\text{constant}$) that improve those in \cite{OrsinaPuel} in some directions can be found in \cite{ACMA2, LiYinKe}. In the first work, the authors prove nonexistence of solution for every $\lambda>0$ small enough provided $\mu\geq p$, while in the second one the authors prove existence of solution for every $\lambda>0$ provided $\mu<\frac{2^*-1-p}{2^*-2}$. We point out that, in both mentioned works, no restriction on $p$ from above is imposed. Nevertheless, if $p\geq 2^*-1$, then the condition $\mu<\frac{2^*-1-p}{2^*-2}$ required by the existence result in \cite{LiYinKe} forces $\mu$ to be negative. We also stress that a blow-up argument is employed in \cite{LiYinKe} in order to obtain a priori estimates, even though the change of unknown \eqref{changeofvariables} is strongly used in order to get rid of a quadratic gradient term from the general problem that the authors consider and, in consequence, non-constant functions $\mu$ cannot be handled with their approach.

Still focusing on problem \eqref{model1} with $\gamma=1$, $\delta>0$ and $\mu\equiv\text{constant}$, the range $\frac{2^*-1-p}{2^*-2}\leq \mu<p$ has not been considered in the literature to our knowledge. However, in this particular situation it is easy to see that the transformation \eqref{changeofvariables} turns \eqref{model1} into a semilinear equation whose nonlinearity presents supercritical growth at infinity and subcritical growth at zero. Therefore, \cite[Theorem 8]{ArcoyaGamezOrsinaPeral} implies (after undoing the change of unknown) that there exists at least a solution to \eqref{model1} for every $\lambda>0$ large enough. Again, last (immediate) result is based on the change of unknown, so does not cover problem \eqref{model1} with nonconstant $\mu$.

On the other hand, the singular case $\delta=0$ has been dealt with recently in \cite{CMS}, one more time for $\mu\equiv\text{constant}>0$ (see also \cite{CarmonaMolinoMoreno, BoccardoOrsinaPorzio} for similar singular problems that involve a nonzero source term). The authors of \cite{CMS} show that, if $\gamma\in (0,1)$, then the situation is similar to the nonsigular case $\delta>0$. Indeed, they prove a nonexistence result for $\lambda>0$ small and an existence result for $\lambda>0$ large. On the contrary, for $\gamma\geq 1$ they prove nonexistence results \emph{for all} $\lambda>0$. This fact exposes the remarkable influence of a strong singularity in the equation. As far as we know, the $\mu\not\equiv\text{constant}$ case for $\delta=0$ has not been studied in the literature.

To sum up, in the present work we aim to develop an approach that permits to deal with $x-$dependent $\mu$ in problem \eqref{model1} and also with singular lower order terms, i.e. $\delta=0$. In order to do so, we will employ topological methods. More precisely, we will find solutions to \eqref{model1} as fixed points of certain compact operator that will be defined in Section~\ref{existence}. The well-definition of such an operator will require the well-posedness of the following problem:

\begin{equation}\label{modelfxprob}
\begin{cases}
\displaystyle-\Delta u+\mu(x)\frac{|\nabla u|^2}{u}=h(x), &x\in\Omega,
\\
u>0, &x\in\Omega,
\\
u=0, &x\in\partial\Omega,
\end{cases}
\end{equation}
where $0\lneq h\in L^q(\Omega)$ for some $q>\frac{N}{2}$.

Singular problems of this kind have risen interest in the recent years. In fact, the existence of solution with $\|\mu\|_{L^\infty(\Omega)}<\frac{1}{2}$ has been proven in \cite{B}, and extended to $\|\mu\|_{L^\infty(\Omega)}<1$ in \cite{MA}. As far as the uniqueness of solution is concerned, some results are known for problems similar to \eqref{modelfxprob}, even though they require either the singularity to be milder or $\mu$ to be constant (see \cite{AS, ACMA, CL}). We will prove that uniqueness for problem \eqref{modelfxprob} holds provided $\|\mu\|_{L^\infty(\Omega)}<1$; the proof is based on a comparison principle that we state in Section~\ref{mainresults} and prove in Section~\ref{existence}. Furthermore, we will show that the condition $\|\mu\|_{L^\infty(\Omega)}<1$ is natural by proving a nonexistence result provided $\mu>1$ in a neighborhood of $\partial\Omega$; the proof follows the ideas in \cite[Lemma 2.5]{CMS}. In next statement we summarize the new uniqueness and nonexistence results that we prove about problem \eqref{modelfxprob} and we include the previously known (\cite{B,MA}) existence part for completeness.

\begin{theorem}\label{fxmodelthm}
Let $0\lneq h\in L^q(\Omega)$ for some $q>\frac{N}{2}$ and let $0\lneq\mu\in L^\infty(\Omega)$. The following statements hold true:
\begin{enumerate}
\item If $\|\mu\|_{L^\infty(\Omega)}<1$, then there exists a unique finite energy solution to problem \eqref{modelfxprob}.
\item \label{item2nonexistence} If there exist an open domain $\omega\subset\subset\Omega$ and a constant $\tau>1$ such that $\mu(x)\geq\tau$ and $h(x)=0$, both for a.e. $x\in\Omega\setminus\omega,$ then problem \eqref{modelfxprob} admits no solution.
\end{enumerate}  
\end{theorem}

We point out that existence and uniqueness results for problem \eqref{modelfxprob} are known for general nonnegative $\mu\in L^\infty(\Omega)$ (i.e. without assuming that $\|\mu\|_{L^\infty(\Omega)}< 1$) provided $h$ is locally bounded away from zero (see \cite{A6} for the existence and \cite{CL} for the uniqueness). Thus, we clarify that the condition $h\equiv 0$ near the boundary in item \eqref{item2nonexistence} of Theorem~\ref{fxmodelthm} is also natural for having nonexistence of solution.  

Once we have shown that problem \eqref{modelfxprob} is well-posed, we will be able to define a compact operator whose fixed points are solutions to \eqref{lambdaproblem} (see Section~\ref{existence}). A version of a result in \cite{Kral} (see \cite{deFigLiNuss}) will assure the existence of a fixed point of the operator. 

As it is mandatory for fixed point theorems, we will prove the existence of a priori estimates on the solutions to a problem related to \eqref{lambdaproblem}. To this task, we will adapt the blow-up method due to \cite{GidasSpruck}. Roughly speaking, this technique consists of assuming by contradiction that there exists a sequence of solutions whose norms blow up as $n$ tends to infinity. The conclusion follows by passing to the limit in a problem satisfied by a certain normalized sequence. In fact, the limit function is a solution to a problem which, however, does not admit any solution by virtue of some Liouville type result. Therefore, one gets a contradiction, so any sequence of solutions must be bounded. In this context, the difficulties that we find are twofold. Firstly, the normalized sequence, say $\{v_n\}$, satisfies an equation having a lower order term of type $\frac{|\nabla v_n|^2}{v_n+\delta_n}$, where $0\leq\delta_n\to 0$ as $n\to+\infty$. If we aim to pass to the limit, then we need to find positive lower bounds on $\{v_n\}$, otherwise the lower order term may blow up as $n\to+\infty$. And lastly, we arrive to a limit problem, having a quadratic gradient lower order term of the form $\frac{|\nabla v|^2}{v}$, for which nonexistence Liouville type results are not known in the literature (a non-exhaustive list of references for Liouville type results about problems depending on the gradient is \cite{FMRS, 
FilippucciPucciSouplet, Lions2, Montoro, PorrettaVeron, SerrinZou, Souplet2}). 

We overcome the first of the difficulties by proving H\"older estimates in spite of the singular quadratic term. The proof follows the ideas of \cite{LU}, which have been widely used for singular problems (see \cite{A6, AM, CLLM, CLM, Durastanti, L, MA}, among others). We will show that these estimates yield in turn positive lower estimates from below and this will be enough to pass to the limit. Regarding the second difficulty, we observe that the limit equation does admit a convenient change of unknown which gets rid of the gradient term, so that we may apply classical Liouville type results (see Section~\ref{estimates} below).

We state here the main existence result for problem \eqref{model1} in the case $\gamma=1$. 

\begin{theorem}\label{modelthm1}
Let $p>1$, $\gamma=1$, $\delta\geq 0$ and $\mu\in L^\infty(\Omega)$. The following statements hold true:
\begin{enumerate}
\item If $\mu\in C(\overline{\Omega})$ and there exist two real numbers $\tau,\sigma$ such that $2\sigma-1<\tau\leq\sigma<\frac{2^*-1-p}{2^*-2}$ and $\tau\leq\mu(x)\leq\sigma$ for all $x\in\overline{\Omega}$, then there exists at least a solution to \eqref{model1} for every $\lambda>0$. If, in addition, either $\mu\geq 0$ or $\delta>0$, then there exists at least a finite energy solution to \eqref{model1} for every $\lambda>0$.
\item If $\delta=0$ and there exist an open domain $\omega\subset\subset\Omega$ and a constant $\tau>1$ such that $\mu(x)\geq\tau$ for a.e. $x\in\Omega\setminus\omega$, then problem \eqref{model1} admits no solution for any $\lambda>0$.
\end{enumerate}  
\end{theorem}

Note that, in the first item of Theorem~\ref{modelthm1}, $\mu$ is allowed to change sign unless  $p\geq 2^*-1$, in which case $\mu$ is necessarily negative. We also point out that the smallness condition $\sigma<\frac{2^*-1-p}{2^*-2}$ in Theorem~\ref{modelthm1} is natural since, in fact, problem \eqref{model1} has no bounded solutions provided $\gamma=1$, $\delta=0$, $\mu\equiv \text{constant}\in\left[\frac{2^*-1-p}{2^*-2},1\right)$ and $\Omega$ is starshaped (see Remark~\ref{pohozaevremark} below). Moreover, we will show later that, strengthening the smallness condition conveniently (in terms of $p,N$), one may assume $\mu$ to be either continuous only in a neighborhood of $\partial\Omega$, or merely bounded in $\Omega$. Also about the existence part of the theorem, it is worth to point out that we need to control $\mu$ from below in order to prove the H\"older estimates that we mentioned. However, if $\mu$ is constant, i.e. $\sigma=\tau$, then the condition $2\sigma-1<\tau$ becomes $\sigma=\tau<1$, which means no restriction since $\sigma<\frac{2^*-1-p}{2^*-2}<1$. On the other hand, we stress that the case $\mu\equiv 1$ and $\delta=0$ remains unsolved, i.e. there are neither existence nor nonexistence results about problem \eqref{model1} for $\gamma=1, \delta=0, \mu\equiv 1, p>1$ in the literature.

Next result shows that our approach allows also to prove existence for $\gamma>1$ and for all $\lambda>0$. 

\begin{theorem}\label{modelthm2}
Let $p>1$, $\gamma>1$, $\delta\geq 0$ and $\mu\in L^\infty(\Omega)$. The following statements hold true:
\begin{enumerate}
\item If $\delta>0$ and $\mu$ satisfies 
\begin{equation}\label{musmallcond}
2\|\mu^+\|_{L^\infty(\Omega)}+\|\mu^-\|_{L^\infty(\Omega)}<\delta^{\gamma-1},
\end{equation}
then there exists at least a finite energy solution to \eqref{model1} for every $\lambda>0$. 
\item If $\delta=0$ and there exist an open domain $\omega\subset\subset\Omega$ and a constant $\tau>0$ such that $\mu(x)\geq\tau$ for a.e. $x\in\Omega\setminus\omega$, then problem \eqref{model1} admits no solution for any $\lambda>0$.
\end{enumerate}
\end{theorem}

Notice that, if $\mu\equiv\text{constant}>0$, then condition \eqref{musmallcond} becomes $\mu<\frac{\delta^{\gamma-1}}{2}$. It is known that such a smallness condition is not needed if $\mu$ is a positive constant (see \cite[Theorem~{1.2}]{OrsinaPuel}). Therefore, we presume that the assumption \eqref{musmallcond} is technical, even though we cannot avoid it since it is used to prove the H\"older estimates mentioned above. Concerning the nonexistence statement for the singular case $\delta=0$, it is proven again following closely \cite[Lemma~{2.5}]{CMS}. We stress that the fact that the singularity is strong, namely $\gamma>1$, allows to take $\mu>0$ near $\partial\Omega$ (in contrast to the case $\gamma=1$, for which $\mu>1$ near the boundary was needed). 

We will be able to go beyond Theorem~\ref{modelthm1} and Theorem~\ref{modelthm2} and deal with any $\gamma>0$. Indeed, for $\gamma\in (0,1)$ and $\delta\geq 0$, we will show that existence of solution holds for every $\lambda>0$ large enough.  We will also deal with $\gamma\geq 1, \delta>0$ and $\mu$ a bounded function with arbitrary size, i.e. we remove the size restrictions on $\mu$ in Theorem~\ref{modelthm1} and Theorem~\ref{modelthm2} as long as $\delta>0$, even though $\lambda$ must be large. The statement of the result is the following.

\begin{theorem}\label{modelthm3}
Let $p>1$, $\gamma>0$, $\delta\geq 0$ and $\mu\in L^\infty(\Omega)$. If $\gamma\geq 1$, assume in addition that $\delta>0$. Assume also one of the following two conditions:
\begin{enumerate}
\item $p<\frac{N+1}{N-1}$, 
\item $p<2^*-1$ and $0\leq\mu\in C(\overline{\Omega})$.
\end{enumerate} 
Then, there exists $\lambda_0>0$ such that there exists at least a solution $u_\lambda$ to \eqref{model1} for every $\lambda>\lambda_0$ satisfying $\lim_{\lambda\to+\infty}\|u_\lambda\|_{L^\infty(\Omega)}=0.$ Moreover, if either $\mu\geq 0$ or $\delta>0$, then $u_\lambda$ is a finite energy solution. 
\end{theorem}

If $\gamma\in (0,1)$, last result is consistent with the known results for $\mu\equiv\text{ constant}>0$ which assure nonexistence for $\lambda>0$ small (see \cite{OrsinaPuel} for $\delta>0$ and \cite{CMS} for $\delta=0$). On the contrary, as we pointed out above, for $\gamma>1$ we would expect an existence result for every $\lambda>0$ (even for $\lambda$ small) without size restrictions on $\mu$. As far as the case $\gamma=1$ and $\delta>0$ is concerned, we ignore whether an existence result for $\lambda>0$ small and general $\mu$ should be expected or not. Two exceptions are the ranges $\mu<\frac{2^*-1-p}{2^*-2}$ and $\mu\geq p$, for which existence (see Theorem~\ref{modelthm1} above and \cite{LiYinKe}) and nonexistence (see \cite{OrsinaPuel, ACMA}) for $\lambda>0$ small are known respectively. In other words, the existence of solution to \eqref{model1} for $\gamma=1$, $\delta>0$, $\frac{2^*-1-p}{2^*-2}\leq \mu<p$ and $\lambda>0$ small remains as an open problem, even for $\mu\equiv\text{ constant}>0$. 

We organize the paper as follows. In Section~\ref{mainresults} we introduce the conditions on $f,g$ and the statements of the general results about problem \eqref{lambdaproblem}. Section~\ref{estimates} is devoted to proving some Liouville type results as well as several propositions that provide the estimates via the blow-up method. In Section~\ref{existence} we prove the results that we state in Section~\ref{mainresults} and in the Introduction. Finally, in the Appendix we gather some technical results that are required throughout the paper.

\vspace{0.5 cm}

\noindent{\bf Acknowledgments.} The problems considered in this work have been proposed by T. Leonori. The research was initiated with his collaboration during one week in Granada and another week in Rome, when most of the ideas contained here emerged. Moreover, for the conclusion of the paper, the interesting ideas, comments and corrections by J. Carmona have also supposed a more than remarkable contribution. In any case, their supervision and support have been essential during the research period. This is why the author wants to warmly thank both, collaborators and friends, for being examples of altruism to follow in this competitive and sometimes hostile world of research. The author wants to thank also D. Ruiz and C. De Coster for their kind and useful suggestions and corrections.


\section{Hypotheses and main results}\label{mainresults}

Let us consider a continuous function $f:[0,+\infty)\to[0,+\infty)$ and a Carath\'eodory function $g:\Omega\times (0,+\infty)\to\mathbb{R}$. We will assume throughout this paper that $g$ satisfies the following condition:
\begin{equation}
\label{ggeneral}
\exists g_0\in C((0,+\infty)): |g(x,s)|\leq g_0(s)\quad\text{a.e. }x\in\Omega,\,\,\forall s> 0.
\end{equation} 
Last hypothesis is essentially the minimal condition that $g$ must satisfy so that the weak formulation of \eqref{lambdaproblem} is well defined:

\begin{definition}\label{conceptofsolution}
Let $f:[0,+\infty)\to[0,+\infty)$ be a continuous function and $g:\Omega\times (0,+\infty)\to\mathbb{R}$ be a Carath\'eodory function satisfying \eqref{ggeneral}. A \emph{solution} to \eqref{lambdaproblem} is a function $u\in H^1_{\mbox{\tiny loc}}(\Omega)\cap L^\infty(\Omega)$ such that
\begin{equation*}
\begin{cases}
& \displaystyle \forall\omega\subset\subset\Omega,\,\,\exists c>0:\quad u(x)\geq c\text{ a.e. }x\in\omega,
\\
&\exists \gamma>0:\quad u^\gamma\in H_0^1(\Omega),\text{ and }
\\
& \displaystyle\int_\Omega\nabla u\nabla \phi+\int_\Omega g(x,u)|\nabla u|^2\phi=\int_\Omega f(u)\phi\quad\forall \phi\in C_c^1(\Omega).
\end{cases} 
\end{equation*}
Besides, we will say that $u$ is a \emph{finite energy solution} to \eqref{lambdaproblem} if it is a solution to \eqref{lambdaproblem} with $\gamma=1$, i.e. if $u\in H_0^1(\Omega)$.
\end{definition}

\begin{remark}\label{testfunctionsremark}
It can be proven by following \cite[Appendix]{CLLM} that, in the previous definition, one can take test functions $\phi\in H_0^1(\Omega)\cap L^\infty(\Omega)$ with compact support. Moreover, if $u$ is a finite energy solution, then one can take any test function belonging to $H_0^1(\Omega)\cap L^\infty(\Omega)$.
\end{remark}

Let us fix $\sigma\in\mathbb{R}$. The following condition, stronger than \eqref{ggeneral}, means a more precise control on $g$ from above:
\begin{equation}\label{sgsmallcond2}
\exists g_0\in C((0,+\infty)):\quad -g_0(s)\leq g(x,s)\leq\frac{\sigma}{s}\quad\text{a.e. }x\in\Omega,\,\,\forall s>0.
\tag{$g_1$}
\end{equation}

Let us consider also the monotonicity condition:
\begin{equation}\label{nondecrasingcond}
s\mapsto sg(x,s)\,\,\text{ is nondecreasing for a.e. }x\in\Omega.
\tag{$g_\nearrow$}
\end{equation}

We state next a comparison principle that will be the key for proving the uniqueness part of Theorem~\ref{fxmodelthm}. The proof follows essentially the arguments in \cite{CLLM}, which in turn are inspired by \cite{ACJT2}.

\begin{theorem}\label{comparison}
Let $0\lneq h\in L^1_{\mbox{\tiny loc}}(\Omega)$ and let $g:\Omega\times (0,+\infty)\to\mathbb{R}$ be a Carath\'eodory function satisfying \eqref{nondecrasingcond} and \eqref{sgsmallcond2} for some $\sigma\in (0,1)$. Let $u, v\in C(\Omega)\cap W^{1,N}_{\mbox{\tiny loc}}(\Omega)$, with $u,v>0$ in $\Omega$, be such that
\begin{align}
\label{compsubsol}
&\int_\Omega \nabla u\nabla\phi+\int_\Omega g(x,u)|\nabla u|^2\phi\leq \int_\Omega h(x)\phi\quad\text{ and }
\\
\label{compsupersol}
&\int_\Omega \nabla v\nabla\phi+\int_\Omega g(x,v)|\nabla v|^2\phi\geq \int_\Omega h(x)\phi
\end{align}
for every $0\leq \phi\in H_0^1(\Omega)\cap L^\infty(\Omega)$ with compact support. 
Suppose also that the following boundary condition holds:
\begin{equation}\label{boundarycond}
\limsup_{x\to x_0}(u(x)^{1-\sigma}-v(x)^{1-\sigma})\leq 0 \quad \forall x_0\in\partial\Omega.
\end{equation}
Then, $u\leq v$ in $\Omega$. 
\end{theorem}

We will present below four existence theorems that represent the core of this work. The proofs are based on a fixed point result. The required a priori estimates are obtained via a blow-up method (see Section~\ref{estimates}).

Let us fix $p>1$ and $\delta\geq 0$. We establish now a growth condition and two limit conditions on $f$ which make it behave like a superlinear power:

\begin{equation}\label{boundsf}
\exists a\ge 1:\quad s^p\leq f(s)\leq a(s+\delta)^p\quad\forall s\geq 0.
\tag{$f_*$}
\end{equation}

\begin{equation}\label{limitf}
\exists L\in (0,+\infty):\quad \lim_{s\to+\infty}\frac{f(s)}{s^p}=L.
\tag{$f_\infty$}
\end{equation}

\begin{equation}\label{fzero}
\lim_{s\to 0} \frac{f(s)}{s}=0.
\tag{$f_0$}
\end{equation}
Observe that \eqref{fzero} implies in particular that $f(0)=0$.

On the other hand, for fixed $\delta\geq 0$ and $\sigma,\tau\in\mathbb{R}$, we set the following growth restriction on $g$ which will allow us to prove certain H\"older estimates (see the Appendix below):

\begin{equation}\label{sgsmallcond}
\begin{cases}
2\sigma-1<\tau\leq\sigma<1,
\\
\tau\leq (s+\delta)g(x,s)\leq\sigma\quad\text{a.e. }x\in\Omega,\,\,\forall s>0.
\end{cases}
\tag{$g_*$}
\end{equation}

We will also need the following limit condition on $g$ at infinity for $p>1$:

\begin{equation}\label{unifconv}
\begin{cases}
&\exists\mu\in C(\overline{\Omega}): \max_{x\in\overline{\Omega}}\mu(x)<\frac{2^*-1-p}{2^*-2},
\\
& \lim_{s\to +\infty}\|s g(\cdot,s)-\mu\|_{L^\infty(\Omega)}=0.
\end{cases}
\tag{$g_\infty$}
\end{equation}

A simple limit condition on $g$ at zero will be required too:
\begin{equation}\label{sgzerolimit}
\exists\lim_{s\to 0} (s+\delta)g(x,s)\quad\text{a.e. }x\in\Omega.
\tag{$g_0$}
\end{equation}
Observe that, if $\delta>0$, then \eqref{sgzerolimit} is equivalent to $g(x,\cdot)$ being continuous at $s=0$ for a.e. $x\in\Omega$. 

We are ready to state our first general existence result:
\begin{theorem}\label{mainthm1}
Let $f:[0,+\infty)\to [0,+\infty)$ be a continuous function and $g:\Omega\times (0,+\infty)\to\mathbb{R}$ be a Carath\'eodory function. For $\delta\geq 0$, $\sigma,\tau\in\mathbb{R}$ and $p>1$, assume that $f$ satisfies \eqref{boundsf}, \eqref{limitf} and \eqref{fzero}, and that $g$ satisfies \eqref{sgsmallcond},\eqref{unifconv} and \eqref{sgzerolimit}.  Then, there exists at least a solution to \eqref{lambdaproblem} for every $\lambda>0$. If, in addition, either $g\geq 0$ or $g(x,\cdot)$ is continuous at $s=0$ a.e. $x\in\Omega$, then there exists at least a finite energy solution to \eqref{lambdaproblem} for every $\lambda>0$.
\end{theorem}

\begin{remark}\label{pohozaevremark}
We point out that the smallness condition $\max_{x\in\overline{\Omega}}\mu(x)<\frac{2^*-1-p}{2^*-2}$ in Theorem~\ref{mainthm1} coming from condition \eqref{unifconv} is necessary for the existence of solutions to \eqref{lambdaproblem}, at least in the model case of $\Omega$ a satarshaped domain, $f(s)=\lambda s^p$ for some $\lambda>0$, $p\in (1,2^*-1)$, and $sg(x,s)\equiv\mu$ for some constant $\mu<1$. Indeed, assume by contradiction that $\mu\in\left[\frac{2^*-1-p}{2^*-2},1\right)$ and that $0<u\in H_0^1(\Omega)\cap L^\infty(\Omega)$ satisfies $-\Delta u+\mu\frac{|\nabla u|^2}{u}=\lambda u^p$ in $\Omega$. Then, $v=c u^{1-\mu}$ satisfies $-\Delta v=v^\frac{p-\mu}{1-\mu}$ in $\Omega$ for some $c>0$. Therefore, since $\frac{p-\mu}{1-\mu}\geq 2^*-1$ and $\Omega$ is starshaped, the well-known Pohozaev's identity (see \cite{Pohozaev}) yields a contradiction.
\end{remark}


Let us fix again $p>1$. Next hypothesis is a limit condition on $g$ at infinity weaker than \eqref{unifconv} as it is only required to hold in a neighborhood of $\partial\Omega$.
\begin{equation}\label{unifconvboundary}
\begin{cases}
&\exists\omega\subset\subset\Omega,\,\,\mu\in C(\overline{\Omega\setminus\omega}): \max_{x\in\overline{\Omega\setminus\omega}}\mu(x)<\frac{2^*-1-p}{2^*-2},
\\
& \lim_{s\to +\infty}\|s g(\cdot,s)-\mu\|_{L^\infty(\Omega\setminus\omega)}=0.
\end{cases}
\tag{$\widetilde{g_\infty}$}
\end{equation}

The following result shows that, assuming a stronger control on $g$ from above, one can relax the limit condition on $g$ at infinity to obtain a solution to \eqref{lambdaproblem}.

\begin{theorem}\label{mainthm2}
Let $f:[0,+\infty)\to [0,+\infty)$ be a continuous function and $g:\Omega\times (0,+\infty)\to\mathbb{R}$ be a Carath\'eodory function. For $\delta\geq 0$, $\sigma,\tau\in\mathbb{R}$ and $p>1$, assume that $f$ satisfies \eqref{boundsf}, \eqref{limitf} and \eqref{fzero}, and that $g$ satisfies \eqref{sgsmallcond}, \eqref{unifconvboundary} and \eqref{sgzerolimit}. Assume in addition that 
$\sigma\leq \frac{N-(N-2)p}{2}$. Then, there exists at least a solution to \eqref{lambdaproblem} for every $\lambda>0$. If, in addition, either $g\geq 0$ or $g(x,\cdot)$ is continuous at $s=0$ a.e. $x\in\Omega$, then there exists at least a finite energy solution to \eqref{lambdaproblem} for every $\lambda>0$.
\end{theorem}

Next theorem shows that the limit conditions at infinity are actually not essential to obtain solutions to \eqref{lambdaproblem}. In return, one has to assume an even stronger control on $g$ from above.

\begin{theorem}\label{mainthm3}
Let $f:[0,+\infty)\to [0,+\infty)$ be a continuous function and $g:\Omega\times (0,+\infty)\to\mathbb{R}$ be a Carath\'eodory function. For $\delta\geq 0$, $\sigma,\tau\in\mathbb{R}$ and $p>1$, assume that $f$ satisfies \eqref{boundsf} and \eqref{fzero}, and that $g$ satisfies \eqref{sgsmallcond} and \eqref{sgzerolimit}. Assume in addition that $\sigma\leq\frac{N+1-(N-1)p}{2}$. Then, there exists at least a solution to \eqref{lambdaproblem} for every $\lambda>0$. Moreover, the following holds.
\begin{enumerate}
\item If either $g\geq 0$ or $g(x,\cdot)$ is continuous at $s=0$ a.e. $x\in\Omega$, then there exists at least a finite energy solution to \eqref{lambdaproblem} for every $\lambda>0$. 
\item If $\delta=0$, then there exists $C>0$ such that, for every $\lambda>0$ and for every solution $u$ to \eqref{lambdaproblem}, the estimate $\lambda^\frac{1}{p-1}\|u\|_{L^\infty(\Omega)}\leq C$ holds.
\end{enumerate}
\end{theorem}

Let $s_0\in (0,1)$, $\sigma,\tau\in\mathbb{R}$ and $p>1$. Consider the following growth conditions near zero:

\begin{equation}\label{boundsftilde}
\exists a\ge 1:\quad s^p\leq f(s)\leq a s^p\quad\forall s\in(0,s_0).
\tag{$\widetilde{f_*}$}
\end{equation}

\begin{equation}\label{sgsmallcondtilde}
\begin{cases}
2\sigma-1<\tau\leq\sigma<1,
\\
\tau\leq s g(x,s)\leq\sigma\quad\text{a.e. }x\in\Omega,\,\,\forall s\in(0,s_0).
\end{cases}
\tag{$\widetilde{g_*}$}
\end{equation}
Notice that condition \eqref{boundsftilde} implies that $\lim_{s\to 0} \frac{f(s)}{s}=0$ and, in particular, $f(0)=0$. On the other hand, we point out that, if $\tau>0$, then condition \eqref{sgsmallcondtilde} forces $g$ to be singular at $s=0$. However, if \eqref{sgsmallcondtilde} holds for some $\tau<0$ and $\sigma>0$, then $g$ may be rather general at $s=0$ since it can be continuous as well as unbounded from above and from below.

Our next existence theorem will require neither growth nor limit conditions at infinity on $f$ and $g$. In exchange, the existence of solution holds only for every $\lambda>0$ large enough.

\begin{theorem}\label{mainthm4}
Let $f:[0,+\infty)\to [0,+\infty)$ be a continuous function and $g:\Omega\times (0,+\infty)\to\mathbb{R}$ be a Carath\'eodory function. For $s_0\in (0,1)$, $\sigma,\tau\in\mathbb{R}$ and $p>1$, assume that $f$ satisfies \eqref{boundsftilde} and that $g$ satisfies \eqref{sgsmallcondtilde}. Assume in addition that there exists $\lim_{s\to 0} sg(x,s)$ for a.e. $x\in\Omega$, and also that $\sigma\leq\frac{N+1-(N-1)p}{2}$. Then, there exists $\lambda_0>0$ such that there exists at least a solution $u_\lambda$ to \eqref{lambdaproblem} for every $\lambda>\lambda_0$ satisfying $\lim_{\lambda\to+\infty}\|u_\lambda\|_{L^\infty(\Omega)}=0$. If, in addition, either $g\geq 0$ or $g(x,\cdot)$ is continuous at $s=0$ a.e. $x\in\Omega$, then $u_\lambda$ is a finite energy solution.
\end{theorem}

Regarding Theorem~\ref{mainthm4}, observe that, if $p>\frac{N+1}{N-1}$, then $\sigma<0$, so $g(x,s)<0$ for $s$ near zero. Thus, for instance, the particular case $g(x,s)=\frac{\mu}{s^\gamma}$ and $f(s)=s^p$ is not covered by Theorem~\ref{mainthm4} if $\gamma\in (0,1)$, $p\in \left[\frac{N+1}{N-1},2^*-1\right)$ and $\mu>0$ is constant. The result that comes next does cover this particular case. In fact, for $p\in (1,2^*-1)$, next theorem allows to consider general continuous functions $g$ satisfying 

\begin{equation}\label{integrablecond}
\exists G\in L^1((0,s_0)):\quad  0\leq g(x,s)\leq G(s)\quad\forall (x,s)\in\Omega\times (0,s_0).
\tag{G}
\end{equation}

\begin{equation}\label{sguniflimitzero}
\begin{cases}
&\exists\mu_0\in C(\overline{\Omega}): \max_{x\in\overline{\Omega}}\mu_0(x)<\frac{2^*-1-p}{2^*-2},
\\
& \lim_{s\to 0}\|s g(\cdot,s)-\mu_0\|_{L^\infty(\Omega)}=0.
\end{cases}
\tag{$\widetilde{g_0}$}
\end{equation}

\begin{remark}\label{muzeroremark}
Notice that \eqref{integrablecond} and \eqref{sguniflimitzero} imply that $\lim_{s\to 0} \|sg(\cdot,s)\|_{L^\infty(\Omega)}=0$, i.e. $\mu_0\equiv 0$ in condition \eqref{sguniflimitzero}. Indeed, we first observe that the fact that $g\geq 0$ implies that
\[\mu_0(x)\geq\mu_0(x)-sg(x,s)\geq-\|sg(\cdot,s)-\mu_0\|_{L^\infty(\Omega)}\quad \forall (x,s)\in\Omega\times (0,s_0).\]
Passing to the limit as $s\to 0$ leads to $\mu_0\geq 0$. On the other hand, we know that for every $\varepsilon>0$ there exists $s_\varepsilon\in (0,1)$ such that 
\[s g(x,s)\geq\mu_0(x)-\varepsilon\quad\forall (x,s)\in \Omega\times(0,s_\varepsilon).\]
Let us assume by contradiction that there exists $x_0\in\Omega$ such that $\mu_0(x_0)>0$. Then we choose $\varepsilon=\frac{\mu_0(x_0)}{2}$ and we obtain
\[G(s)\geq g(x_0,s)\geq\frac{\varepsilon}{s}\quad\forall s\in (0,s_\varepsilon).\]
This contradicts hypothesis \eqref{integrablecond}.
\end{remark}

The result reads as follows:
\begin{theorem}\label{mainthm5}
Let $g:\overline{\Omega}\times (0,+\infty)\to\mathbb{R}$ be a continuous function. For $s_0\in (0,1)$ and $p\in (1,2^*-1)$, assume that $g$ satisfies \eqref{integrablecond} and \eqref{sguniflimitzero}. Let us also consider the function $f:[0,+\infty)\to [0+\infty)$ defined by $f(s)=s^p$ for all $s\geq 0$. Then, there exists $\lambda_0>0$ such that there exists at least a finite energy solution $u_\lambda$ to \eqref{lambdaproblem} for every $\lambda>\lambda_0$ satisfying $\lim_{\lambda\to+\infty}\|u_\lambda\|_{L^\infty(\Omega)}=0$.
\end{theorem}

\begin{remark}
We point out that Theorem~\ref{mainthm5} is valid for a very wide class of functions $g$. For instance, if
\[g(x,s)=\mu(x)h(s)\quad\forall (x,s)\in\Omega\times [0,+\infty),\]
where $0\lneq\mu\in C(\overline{\Omega})$ and $h:[0,+\infty)\to [0,+\infty)$ is continuous (also at $s=0$), then $g$ satisfies \eqref{integrablecond} and \eqref{sguniflimitzero}. On the other hand, a prototypical example of function $g$ singular at $s=0$ satisfying the conditions \eqref{integrablecond} and \eqref{sguniflimitzero} is
\[g(x,s)=\frac{\mu(x)}{s^\gamma}\quad\forall (x,s)\in\Omega\times (0,+\infty)\]
where $\gamma\in (0,1)$, $0\lneq\mu\in C(\overline{\Omega})$.
\end{remark}

The last theorem of the section will be concerned with the following problem:
\begin{equation}\label{hproblem}
\begin{cases}
-\Delta u+g(x,u)|\nabla u|^2=h(x), &x\in\Omega,
\\
u>0, &x\in\Omega,
\\
u=0, &x\in\partial\Omega.
\end{cases}
\tag{$H$}
\end{equation}
The result provides a necessary condition for the existence of solution to \eqref{hproblem}. In order to prove it, we will assume, for fixed $\tau>0$, that $g(x,u)$ is sufficiently large near $\partial\Omega$ in the sense of the following condition:
\begin{equation}\label{gtoolarge}
\exists \omega\subset\subset\Omega,\,\, s_0\in (0,1):\quad sg(x,s)\geq\tau\quad\text{a.e. }x\in\Omega\setminus\omega,\,\,\forall s\in (0,s_0).
\tag{$g_2$}
\end{equation}
The theorem states the following:
\begin{theorem}\label{nonexistence}
Let $h\in L^1(\Omega)$ and let $g:\Omega\times (0,+\infty)\to\mathbb{R}$  be a Carath\'eodory function satisfying \eqref{ggeneral} and \eqref{gtoolarge} for some $\tau>1$. Then, every solution $u$ to \eqref{hproblem} satisfies
\[\int_\Omega\frac{|h(x)|}{u}=+\infty.\]
\end{theorem}


\section{A priori estimates}\label{estimates}

We begin the section with two Liouville type results which will be the key points for proving a priori estimates. The first of them is the following.

\begin{lemma}\label{liouvillelemma}
Let $p>1$ and let $h:(0,+\infty)\to[0,+\infty)$ be a continuous function satisfying
\begin{equation*}
sh(s)\leq \sigma\quad\forall s>0
\end{equation*}
for some $\sigma\in(0,1)$. Let us assume that there exists $\alpha>0$ such that the function $\psi:[0,+\infty)\to [0,+\infty)$, defined by
\[\psi(s)=\int_0^s e^{-\int_\alpha^t h(r)dr}dt\quad\forall s\geq 0,\]
satisfies that 
\[s\mapsto\frac{\psi'(s)s^p}{\psi(s)^{2^*-1}}\quad\text{ is decreasing for all } s>0.\]
Then, the following problem
\begin{equation}\label{Xhprob}
\begin{cases}
-\Delta u+h(u)|\nabla u|^2=u^p, &x\in X,
\\
u>0, &x\in X,
\\
u=0, &x\in\partial X,
\end{cases}
\end{equation}
admits no solutions in $H^1_{\mbox{\tiny loc}}(X)\cap C(\overline{X})$, where $X$ denotes either $\mathbb{R}^N$ or $\mathbb{R}^N_+$.
\end{lemma}

\begin{remark}\label{liouvilleremark}
We stress that Lemma~\ref{liouvillelemma} includes the particular case $h(s)=\frac{\sigma}{s}$ with $\sigma<\frac{2^*-1-p}{2^*-2}$. Indeed, in this case, it is easy to check that $\frac{\psi'(s)s^p}{\psi(s)^{2^*-1}}=c s^{p-\sigma-(1-\sigma)(2^*-1)}$ for some constant $c>0$. Thus, it is decreasing if, and only if, $\sigma<\frac{2^*-1-p}{2^*-2}$. Moreover, $\frac{2^*-1-p}{2^*-2}<1$, so $sh(s)=\sigma<1$ for all $s>0$.
\end{remark}

\begin{proof}[Proof of Lemma~\ref{liouvillelemma}]
Arguing by contradiction, assume that there exists a solution $u\in H^1_{\mbox{\tiny loc}}(X)\cap C(\overline{X})$ to \eqref{Xhprob}. Straightforward computations imply that $v=\psi(u)\in H^1_{\mbox{\tiny loc}}(X)\cap C(\overline{X})$ satisfies
\begin{equation}\label{changedXprob}
\begin{cases}
-\Delta v=\varphi(v), &x\in X,
\\
v>0, &x\in X,
\\
v=0, &x\in\partial X,
\end{cases}
\end{equation}
where $\varphi(t)=\psi'(\psi^{-1}(t))\psi^{-1}(t)^p$ for all $t\in \text{Im}(\psi)$. Moreover, classical elliptic regularity theory implies that $v\in C^2(X)$. We claim now that, actually, problem \eqref{changedXprob} admits no solutions $v\in C^2(X)\cap C(\overline{X})$.

Indeed, one can easily deduce that
\[\psi(s)\geq\frac{\alpha^\sigma \left(s^{1-\sigma}-\alpha^{1-\sigma}\right)}{1-\sigma}\quad\forall s>\alpha.\]
Hence, $\lim_{s\to +\infty}\psi(s)=+\infty$. In consequence, the function $\varphi$ is defined in $[0,+\infty)$. On the other hand, it is clear that the function $t\mapsto\frac{\varphi(t)}{t^{2^*-1}}$ is decreasing for all $t>0$. Therefore, in case $X=\mathbb{R}^N$, the claim follows from \cite[Theorem 3]{Bianchi}.

Notice also that $\lim_{t\to 0}\frac{\varphi(t)}{t}=0$ and $\varphi(t)\geq c\psi^{-1}(t)^{p-\sigma}$ for every $t>0$ and for some $c>0$. Hence, the claim also holds true in case $X=\mathbb{R}^N_+$ by virtue of \cite[Theorem 1.3]{DamGlad}.

In any case, we have arrived to a contradiction. The proof is now completed.
\end{proof}

We present here the second Liouville type result which is valid for supersolutions.

\begin{lemma}\label{liouvillelemmasupersol}
Let $p>1$ and $\sigma\leq \frac{N-(N-2)p}{2}$. Then, the following problem
\begin{equation}\label{Xprob}
\begin{cases}
\displaystyle-\Delta u+\sigma\frac{|\nabla u|^2}{u}=u^p, &x\in X,
\\
u>0, &x\in X,
\\
u=0, &x\in\partial X,
\end{cases}
\end{equation}
admits no supersolutions in $H^1_{\mbox{\tiny loc}}(X)\cap C(\overline{X})$ provided $X=\mathbb{R}^N$ . On the other hand, if we assume that $\sigma\leq \frac{N+1-(N-1)p}{2}$, then problem \eqref{Xprob} with $X=\mathbb{R}^N_+$ admits no supersolutions in $H^1_{\mbox{\tiny loc}}(X)\cap C(\overline{X})$.
\end{lemma}

\begin{remark}
Note that $\frac{N+1-(N-1)p}{2}<\frac{N-(N-2)p}{2}<\frac{2^*-1-p}{2^*-2}<1$, so the smallness conditions on $\sigma$ in Remark~\ref{liouvilleremark} and in Lemma~\ref{liouvillelemmasupersol} are gradually more restrictive. We also stress that such conditions on $\sigma$ in Lemma~\ref{liouvillelemmasupersol} are sharp. Indeed, if $\sigma>\frac{N-(N-2)p}{2}$ (resp. $\sigma>\frac{N+1-(N-1)p}{2}$), then one can find explicit supersolutions to \eqref{Xprob} for $X=\mathbb{R}^N$, see \cite{MitPoho} (resp. $X=\mathbb{R}^N_+$, see \cite{BirMit}). 
\end{remark}


\begin{proof}[Proof of Lemma~\ref{liouvillelemmasupersol}]
Reasoning by contradiction, assume that there exists $u\in H^1_{\mbox{\tiny loc}}(X)\cap C(\overline{X})$ a supersolution to \eqref{Xprob}. Then, there is a constant $c>0$ such that $v=c u^{1-\sigma}\in H^1_{\mbox{\tiny loc}}(X)\cap C(\overline{X})$ is a supersolution to 
\begin{equation*}
\begin{cases}
-\Delta v=v^\frac{p-\sigma}{1-\sigma}, &x\in X,
\\
v>0, &x\in X,
\\
v=0, &x\in\partial X.
\end{cases}
\end{equation*}
Hence, if $\sigma\leq\frac{N-(N-2)p}{2}$, then $\frac{p-\sigma}{1-\sigma}\leq \frac{N}{N-2}$, so in case $X=\mathbb{R}^N$ we arrived to a contraction with \cite[Theorem 2.1]{MitPoho}. On the other hand, if $\sigma<\frac{N+1-(N-1)p}{2}$, then $\frac{p-\sigma}{1-\sigma}<\frac{N+1}{N-1}$, so we have again a contradiction with \cite[Theorem 3.1]{BirMit}.
\end{proof}

Let $t\geq 0$, $\lambda>0$, $f:[0,+\infty)\to [0,+\infty)$ be a continuous function satisfying \eqref{boundsf} for some $p>1$ and $g:\Omega\times (0,+\infty)\to [0,+\infty)$ be a Carath\'eodory function satisfying \eqref{sgsmallcond2} for some $\sigma\in (0,1)$. From now and up to the end of the section, we will restrict ourselves to $g\geq 0$ and we will also impose diverse upper bounds on $p$. However, we will show in Section~\ref{existence} that, in most cases, those restrictions can be relaxed for proving existence of solution. 

Let us consider the following auxiliary problem:
\begin{equation}\label{tproblem}
\begin{cases}
\dys-\Delta u+g(x,u)|\nabla u|^2= \lambda f(u)+t u^\sigma, &x\in\Omega,
\\
u> 0, &x\in\Omega,
\\
u=0, &x\in\partial\Omega.
\end{cases}
\tag{$P^t$}
\end{equation}
Note that, in the term $tu^\sigma$, the exponent $\sigma$ is the same number that appears in condition \eqref{sgsmallcond2}. We will derive a priori estimates on the solutions to \eqref{tproblem} which will provide the existence of solution to \eqref{lambdaproblem}.

Next proposition gives an a priori estimate on the parameter $t$ in problem \eqref{tproblem}. 

\begin{proposition}\label{testimate}
Let $\lambda>0$, $f:[0,+\infty)\to [0,+\infty)$ be a continuous function and $g:\Omega\times (0,+\infty)\to[0,+\infty)$ be a Carath\'eodory function. For $p>1$ and $\sigma\in (0,1)$, assume that $f$ satisfies  \eqref{boundsf} and $g$ satisfies \eqref{sgsmallcond2}. Then, there exists $t_0>0$ such that problem \eqref{tproblem} admits no solution for any $t>t_0$.
\end{proposition}

\begin{proof}
Let $u$ be a solution to \eqref{tproblem} for some $t>0$. For a fixed smooth open set $\omega\subset\subset\Omega$, let $\lambda_1$ be the principal eigenvalue to the homogeneous Dirichlet eigenvalue problem in $\omega$, and let $\varphi_1$ be any positive associated eigenfunction, i.e. $\lambda_1$ and $\varphi_1$ satisfy
\begin{equation*}
\begin{cases}
-\Delta\varphi_1=\lambda_1\varphi_1,& x\in\omega,
\\
\varphi_1>0,& x\in\omega
\\
\varphi_1=0,& x\in\partial\omega.
\end{cases}
\end{equation*}
If we extend $\varphi_1\equiv 0$ in $\Omega\setminus\omega$, then the function $\phi=\frac{\varphi_1}{u^\sigma}$ belongs to $H_0^1(\Omega)\cap L^\infty(\Omega)$ and has compact support. Taking $\phi$ as test function in \eqref{tproblem} (which is allowed by virtue of Remark~\ref{testfunctionsremark}) we obtain
\begin{equation}\label{testimate1}
\int_\Omega\frac{\nabla u\nabla \varphi_1}{u^\sigma}-\sigma\int_\Omega\frac{\varphi_1|\nabla u|^2}{u^{\sigma+1}}+\int_\Omega g(x,u)\frac{\varphi_1|\nabla u|^2}{u^\sigma}=\lambda\int_\Omega \frac{f(u)}{u^\sigma}\varphi_1+t\int_\Omega\varphi_1.
\end{equation}
On the one hand, it is clear by \eqref{sgsmallcond2} that
\[-\sigma\int_\Omega\frac{\varphi_1|\nabla u|^2}{u^{\sigma+1}}+\int_\Omega g(x,u)\frac{\varphi_1|\nabla u|^2}{u^\sigma}\leq 0.\]
Thus, using also \eqref{boundsf} we deduce from \eqref{testimate1} that 
\begin{equation}\label{testimate2}
t\int_\Omega\varphi_1+\lambda\int_\Omega u^{p-\sigma}\varphi_1\leq\int_\Omega\frac{\nabla u\nabla\varphi_1}{u^\sigma}.
\end{equation}
On the other hand, let us denote as $\nu$ the exterior normal unit vector to $\partial\omega$. Then, Hopf's lemma implies that $\nu\nabla\varphi_1<0$ on $\partial\omega$. Hence, integration by parts in $\int_\omega(-\Delta\varphi_1) u^{1-\sigma}$ and Young's inequality yield
\[\int_\Omega\frac{\nabla u\nabla\varphi_1}{u^\sigma}=\lambda_1\int_\omega\frac{\varphi_1 u^{1-\sigma}}{1-\sigma}+\int_{\partial\omega}\frac{u^{1-\sigma}}{1-\sigma}\nu\nabla\varphi_1<\lambda_1\int_\omega\frac{\varphi_1 u^{1-\sigma}}{1-\sigma}\leq\frac{\lambda}{2}\int_\Omega u^{p-\sigma}\varphi_1+C.\]
In sum, from \eqref{testimate2} we deduce that $t\leq t_0$ for some $t_0>0$, as we wanted to prove.
\end{proof}

Next lemma provides a summability property that the solutions to \eqref{tproblem} satisfy. The interesting point is that such a property becomes better as the singularity of $g$ becomes stronger. 

\begin{lemma}\label{summabilitylemma}
Let $\lambda>0$, $\sigma\in (0,1)$, $f:[0,+\infty)\to [0,+\infty)$ be a continuous function and $g:\Omega\times (0,+\infty)\to [0,+\infty)$ be a Carath\'eodory function satisfying \eqref{ggeneral}. For some $a>0$, $p\geq 1$ and $\tau\in [0,1)$, assume that
\begin{align*}
f(s)&\leq a s^p \quad \forall s\geq 0,
\\
sg(x,s)&\geq \tau \quad \text{a.e. }x\in\Omega,\,\,\forall s>0.
\end{align*} 
Let $u\in H^1_{\mbox{\tiny loc}}(\Omega)\cap L^\infty(\Omega)$ be a solution to \eqref{tproblem} for some $t\geq 0$. Then, $u^\gamma\in H_0^1(\Omega)$ for every $\gamma> \frac{\max\{1-\sigma,1-\tau\}}{2}$.
\end{lemma}

\begin{proof}
First of all, recall that Lemma~\ref{reglemma} in the Appendix below implies that $u\in C(\overline{\Omega})$. Now, for any $\alpha>0$, let us consider the function $v=u^\alpha\in H^1_{\mbox{\tiny loc}}(\Omega)\cap C(\overline{\Omega})$. It is easy to see that $v$ satisfies
\begin{align*}
-\Delta v &=\frac{(1-\alpha-v^\frac{1}{\alpha}g(x,v^\frac{1}{\alpha})) |\nabla v|^2}{\alpha v} + \alpha\lambda v^\frac{\alpha-1}{\alpha} f(v^\frac{1}{\alpha})+\alpha t v^{\frac{\alpha-1+\sigma}{\alpha}}
\\
&\leq \frac{(1-\alpha-\tau)|\nabla v|^2}{\alpha v}+\alpha \lambda a v^\frac{\alpha-1+p}{\alpha}+\alpha t v^{\frac{\alpha-1+\sigma}{\alpha}}, \quad x\in\Omega.
\end{align*}
Moreover, if we take $\alpha\geq\max\{1-\sigma,1-\tau\}$, then
\begin{equation}\label{summabilityeq}
-\Delta v\leq C,\quad x\in\Omega,
\end{equation}
for some constant $C>0$.

For every $\varepsilon,\beta>0$, let us now consider the function $\phi=(v^\beta-\varepsilon)^+$. It is clear that $\phi\in H^1_{\mbox{\tiny loc}}(\Omega)\cap C(\overline{\Omega})$. Furthermore, the continuity of $v$ up to $\partial\Omega$ implies that $\phi$ has compact support in $\Omega$. In sum, $\phi\in H_0^1(\Omega)\cap C(\overline{\Omega})$ and has compact support. Therefore, even though $v$ might not belong to $H_0^1(\Omega)$, it follows from Remark~\ref{testfunctionsremark} that one can take $\phi$ as test function in \eqref{summabilityeq}. Thus, we obtain
\[\int_\Omega \chi_{\{v^\beta\geq\varepsilon\}} \frac{|\nabla v|^2}{v^{1-\beta}}\leq C,\]
for another constant $C>0$ independent of $\varepsilon$. Now we let $\varepsilon$ tend to zero and, by virtue of Fatou's lemma, we deduce
\[\int_\Omega\frac{|\nabla v|^2}{v^{1-\beta}}\leq C.\]
Taking into account that \[\frac{|\nabla v|^2}{v^{1-\beta}}= C \left|\nabla v^\frac{1+\beta}{2}\right|^2=C\left|\nabla u^{\alpha\frac{1+\beta}{2}}\right|^2,\]
then we have proved that $u^{\alpha\frac{1+\beta}{2}}\in H_0^1(\Omega)$ for every $\alpha\geq\max\{1-\sigma, 1-\tau\}$ and every $\beta>0$. This proves the result.
\end{proof}

In the following result we prove a priori estimates on the solutions to \eqref{tproblem} if $f$ satisfies \eqref{boundsf} and \eqref{limitf} and $g$ satisfies \eqref{sgsmallcond} and \eqref{unifconv}. In the proof we adapt the blow-up method due to \cite{GidasSpruck}.

\begin{proposition}\label{gidasspruck}
Let $\lambda>0$, $f:[0,+\infty)\to[0,+\infty)$ be a continuous function and $g:\Omega\times (0,+\infty)\to[0,+\infty)$ be a Carath\'edory function. For $\delta,\tau\geq 0$, $\sigma>0$ and $p\in (1,2^*-1)$, assume that $f$ satisfies \eqref{boundsf} and \eqref{limitf}, and that $g$ satisfies \eqref{sgsmallcond} and \eqref{unifconv}. Then, there exists $C>0$ such that $\|u\|_{L^\infty(\Omega)}\leq C$ for every solution $u$ to \eqref{tproblem} for all $t\in [0,t_0]$, where $t_0>0$ is given by Proposition~\ref{testimate}.
\end{proposition}

\begin{proof}
Reasoning by contradiction, we assume that there exist two sequences $\{t_n\}\subset [0,t_0]$ and $\{u_n\}$ such that $u_n$ is a solution to $(P^{t_n})$ for all $n$ and $\|u_n\|_{L^\infty(\Omega)}\to+\infty$ as $n\to+\infty$. By virtue of Lemma~\ref{continuouslemma}, we may consider a sequence $\{x_n\}\subset\Omega$ satisfying
\[\|u_n\|_{L^\infty(\Omega)}=u_n(x_n)\quad\forall n,\quad x_n\to x_0\in\overline{\Omega}\text{, up to a subsequence}.\]
We divide the rest of the proof into two parts. In the first of them we consider the case $x_0\in\Omega$, while the second one is devoted to the case $x_0\in\partial\Omega$. In turn, we divide each part into several steps.

\noindent{\bf CASE 1) $x_0\in\Omega$.}

\noindent{\bf Step 1.1) Scaling.}

Denoting $d=\text{dist}(x_0,\partial\Omega)/2>0$ and $\eta_n=\|u_n\|_{L^\infty(\Omega)}^{-\frac{p-1}{2}}$, we define $v_n:B_{d/\eta_n}(0)\to[0,+\infty)$ by
\[v_n(y)=\eta_n^{\frac{2}{p-1}}u_n(x_n+\eta_n y)\quad\forall y\in B_{d/\eta_n}(0).\]
Therefore, $v_n\in H^1(B_{d/\eta_n}(0))\cap L^\infty(B_{d/\eta_n}(0))$ and satisfies the equation
\begin{equation}\label{vneq}
-\Delta v_n+u_n g_n(y,u_n)\frac{|\nabla v_n|^2}{v_n}=\lambda v_n^p\frac{f(u_n)}{u_n^p}+t_n\eta_n^{\frac{2(p-\sigma)}{p-1}} v_n^{\sigma},\quad y\in B_{d/\eta_n}(0),
\end{equation}
where $g_n(y,s)=g(x_n+\eta_n y,s)$ for a.e. $y\in B_{d/\eta_n}(0)$ and $s>0$. Moreover, $\|v_n\|_{L^\infty(B_{d/\eta_n}(0))}=v_n(0)=1.$ Our aim now is to pass to the limit in \eqref{vneq}. In the next step we will prove the a priori estimates that will provide such a limit.

\noindent{\bf Step 1.2) A priori estimates.}

Let us fix $R>0$ and denote $\omega=B_R(0)$. It is clear that $\omega\subset B_{d/\eta_n}(0)$ for every $n$ sufficiently large, so $v_n$ satisfies the equation \eqref{vneq} in $\omega$ and $\|v_n\|_{L^\infty(\omega)}=1$ for $n$ large. Of course, the same thing happens in $B_{2R}(0)$. Notice also that, if $\delta=0$, then Lemma~\ref{summabilitylemma} implies that $u_n^\gamma\in H_0^1(\Omega)$ for all $\gamma\in\left(\frac{\max\{1-\sigma,1-\tau\}}{2},1-\sigma\right]=\left(\frac{1-\tau}{2},1-\sigma\right]$, where this last equality follows from $\tau<\sigma$ in condition \eqref{sgsmallcond}. Therefore, $v_n^\gamma\in H^1(B_{2R}(0))$ for all $\gamma\in \left(\frac{1-\tau}{2},1-\sigma\right]$. This fact, together with conditions \eqref{sgsmallcond} and  \eqref{boundsf}, allow to apply Lemma~\ref{holderestimatelemma} in the Appendix below to deduce that there exist $C>0$, $\alpha\in (0,1)$ such that 
\begin{equation*}
\|v_n\|_{C^{0,\alpha}(\overline\omega)}\leq C
\end{equation*}
for every $n$ large enough. As a consequence, there exists $v\in C(\overline\omega)$ such that, up to a subsequence, \[v_n\to v\text{ uniformly in }\overline\omega.\] Observe that $\|v\|_{L^\infty(\omega)}=1$ so, in particular, $v\not\equiv 0$.

On the other hand, let us consider a function $\phi\in C^1_c(B_{2R}(0))$ such that $0\leq\phi\leq 1$ in $B_{2R}(0)$ and $\phi\equiv 1$ in $\omega$. Now we multiply both sides of \eqref{vneq} by $v_n\phi^2\in H^1_0(B_{2R}(0))\cap L^\infty(B_{2R}(0))$ and integrate by parts, obtaining
\[\int_{B_{2R}(0)}|\nabla v_n|^2\phi^2+2\int_{B_{2R}(0)}v_n\phi\nabla v_n\nabla\phi\leq C,\]
where we have used that $g\geq 0$, $\|v_n\|_{L^\infty(B_{2R}(0))}=1$ and, one more time, condition \eqref{boundsf}. Hence, by Young's inequality we easily deduce that  
\begin{equation*}
\int_\omega|\nabla v_n|^2\leq C\left(\int_{B_{2R}(0)}|\nabla \phi|^2 v_n^2+1\right)\leq C.
\end{equation*}
That is to say, $\|v_n\|_{H^1(\omega)}\leq C$, and then, up to a subsequence, \[v_n\rightharpoonup v\text{ weakly in }H^1(\omega).\]

We will prove next that, for all $\omega_0\subset\subset\omega$, $v_n$ is bounded from below in $\overline{\omega_0}$ by a positive constant independent of $n$. The approach by comparison due to \cite{MA} is valid here. Indeed, it is straightforward to see that the function $w_n=\frac{v_n^{1-\sigma}}{1-\sigma}\in H^1(\omega)\cap L^\infty(\omega)$ satisfies
\[\int_{\omega} \nabla w_n \nabla\phi= \lambda\int_{\omega} v_n^{p-\sigma}\frac{f(u_n)}{u_n^p}\phi+\int_\omega(\sigma-u_n g_n(y,u_n))\frac{|\nabla v_n|^2\phi}{v_n^{\sigma+1}}+t_n\eta_n^\frac{2(p-\sigma)}{p-1}\int_\omega\phi\]
for all $\phi\in C_c^1(\omega)$. Therefore, using conditions \eqref{sgsmallcond} and \eqref{boundsf}, we derive
\[\int_{\omega} \nabla w_n \nabla\phi\geq \lambda\int_{\omega} v_n^{p-\sigma}\phi\]
for all $0\leq\phi\in C_c^1(\omega)$. On the other hand, let $z_n\in H_0^1(\omega)\cap C(\overline{\omega})$ be the unique solution to
\[
\begin{cases}
\displaystyle-\Delta z_n=\lambda v_n^{p-\sigma}, &y\in\omega,
\\
z_n=0, &y\in\partial\omega.
\end{cases}
\]
It is clear that $z_n\to z$ uniformly in $\overline{\omega}$ and $z_n\rightharpoonup z$ weakly in $H_0^1(\omega)$ for some $z\in H_0^1(\omega)\cap C(\overline{\omega})$. In consequence, $z$ satisfies
\[
\begin{cases}
\displaystyle-\Delta z=\lambda v^{p-\sigma}, &y\in\omega,
\\
z=0, &y\in\partial\omega.
\end{cases}
\]
Since $v\gneq 0$ in $\omega$, the strong maximum principle implies that, for every $\omega_0\subset\subset\omega$, there exists $c>0$ such that $z\geq c$ in $\omega_0$. Besides, by comparison, $w_n\geq z_n$ in $\overline{\omega}$. Hence, the uniform convergence of $\{z_n\}$ implies that, for every $\varepsilon>0$, we may take $n$ large enough so that $v_n^{1-\sigma}\geq(1-\sigma)(z-\varepsilon)$ in $\omega$. In sum, by choosing $\varepsilon=\frac{c}{2}$ we conclude that

\begin{equation}\label{lowerlocalestimate}
\forall\omega_0\subset\subset\omega,\,\,\exists c_{\omega_0}>0:\quad v_n\geq c_{\omega_0}, \quad y\in\omega_0,\,\,\forall n\text{ large.}
\end{equation}

From the previous estimates, it is straightforward to prove that $\{\Delta v_n\}$ is bounded in $L^1_{\mbox{\tiny loc}}(\omega)$. Then, \cite{BoccMur} implies that, passing to a subsequence, 
\[\nabla v_n\to\nabla v,\quad y\in\omega.\]

We are ready now to pass to the limit in \eqref{vneq}. 

\noindent{\bf Step 1.3) Passing to the limit.}

We already know that $v_n\geq c_{\omega_0}>0$ in $\omega_0\subset\subset\omega$. Moreover, $\|u_n\|_{L^\infty(\omega)}\to +\infty$. Therefore, $u_n\to+\infty$ locally uniformly in $\omega$. Then, by \eqref{unifconv} we deduce that $|u_n g_n(y,u_n)-\mu(x_n+\eta_n y)|\to 0$ locally uniformly in $\omega$. In consequence, the continuity of $\mu$ yields
\[u_n g_n(y,u_n)\to\mu(x_0)\quad\text{ locally uniformly in }\omega.\] 
In sum, we have that
\[0\leq u_n g_n(y,u_n)\frac{|\nabla v_n|^2}{v_n}\to\mu(x_0)\frac{|\nabla v|^2}{v}\quad\text{ pointwise in }\omega.\]
Furthermore, it follows from conditions \eqref{limitf} and \eqref{boundsf} and from the Dominated Convergence Theorem that 
\[\int_\omega v_n^p\frac{f(u_n)}{u_n^p}\phi\to\int_\omega v^p\phi\quad\forall\phi\in C_c^1(\omega).\]
Let us take $\phi\in C_c^1(\omega)$ such that $\phi\geq 0$ as test function in the weak formulation of \eqref{vneq}. By virtue of Fatou's lemma and using the convergences that we have proved, it is immediate to show that
\[\int_\omega\nabla v\nabla\phi+\mu(x_0)\int_\omega\frac{|\nabla v|^2}{v}\phi\leq \lambda\int_\omega v^p\phi.\]
If we take now $\frac{v\phi}{v_n}\in H_0^1(\omega)\cap L^\infty(\omega)$ as test function in \eqref{vneq}, we obtain
\begin{align*}
\int_\omega\frac{\nabla v_n\nabla v}{v_n}\phi &-\int_\omega\frac{v|\nabla v_n|^2}{v_n^2}\phi+\int_\omega\frac{v}{v_n}\nabla v_n\nabla\phi+\int_\omega u_n g_n(y,u_n)\frac{v|\nabla v_n|^2}{v_n^2}\phi
\\
&=\lambda\int_\omega v_n^{p-1}v\frac{f(u_n)}{u_n^p}\phi+t_n\eta_n^\frac{2(p-\sigma)}{p-1}\int_\omega\frac{v\phi}{v_n^{1-\sigma}}.
\end{align*}
Observe that
\[0\leq (1-\sigma)\frac{v|\nabla v_n|^2\phi}{v_n^2}\leq (1-u_n g_n(y,u_n))\frac{v|\nabla v_n|^2\phi}{v_n^2}\to(1-\mu(x_0))\frac{|\nabla v|^2}{v}\phi.\]
Then, again by Fatou's lemma we derive 
\[\int_\omega\nabla v\nabla\phi+\mu(x_0)\int_\omega\frac{|\nabla v|^2}{v}\phi\geq \lambda\int_\omega v^p\phi.\]
That is to say, $v\in H^1(\omega)\cap C(\overline{\omega})$ satisfies 
\[-\Delta v+\mu(x_0)\frac{|\nabla v|^2}{v}=\lambda v^p,\quad y\in\omega.\]

\noindent{\bf Step 1.4) Conclusion.}

Since $\omega=B_R(0)$ for arbitrary $R>0$, then a standard diagonal argument (see \cite{GidasSpruck}) implies that $v$ is well-defined in $\mathbb{R}^N$, it belongs to $H^1_{\mbox{\tiny loc}}(\mathbb{R}^N)\cap C(\mathbb{R}^N)$ and it satisfies
\[-\Delta v+\mu(x_0)\frac{|\nabla v|^2}{v}=\lambda v^p,\quad y\in\mathbb{R}^N.\]
Furthermore, it is straightforward to check that the function $w=\lambda^{\frac{1}{p-1}}v$ satisfies 
\[-\Delta w+\mu(x_0)\frac{|\nabla w|^2}{w}=w^p,\quad y\in\mathbb{R}^N.\]
This is impossible by virtue of Lemma~\ref{liouvillelemma} (see also Remark~\ref{liouvilleremark}).

\noindent{\bf CASE 2) $x_0\in\partial\Omega$.}

\noindent{\bf Step 2.1) Scaling.}

Recall that we are assuming that there exist sequences $\{t_n\}\subset[0,t_0]$, $\{u_n\}$ and $\{x_n\}\subset\Omega$ such that $u_n$ is a solution to $(P^{t_n})$ for all $n$, $\|u_n\|_{L^\infty(\Omega)}=u_n(x_n)\to +\infty$ as $n\to+\infty$ and $x_n\to x_0$ for some $x_0\in\partial\Omega$. Taking advantage of the smoothness of $\partial\Omega$, we are allowed to perform a convenient change of coordinates in such a way that $u_n$ is a solution to a similar problem except that $\partial\Omega$ becomes flat near $x_0$ (see Lemma~\ref{straighteninglemma} in the Appendix below for the detailed proof). In other words, we may assume without loss of generality that $u_n$ is a solution to $(R^{t_n})$ for all $n$, where
\begin{equation}
\begin{cases}
-\div(M(x)\nabla v)+b(x)\nabla v+g(x,v)M(x)\nabla v\nabla v=\lambda f(v)+tv^\sigma,\quad &x\in \Omega,
\\
v>0,\quad &x\in \Omega,
\\
v=0,\quad &x\in \Gamma,
\end{cases}
\tag{$R^t$}
\end{equation}
with $\Omega\subset\mathbb{R}^N_+$, $\emptyset\not=\Gamma\subset\partial\Omega\cap\partial\mathbb{R}^N_+$ being open as a subset of $\partial\mathbb{R}^N_+$ and connected, $M\in C^1(\overline{\Omega})^{N\times N}$ being uniformly elliptic, and $b\in C(\overline{\Omega})^N$.

Since $\Gamma$ is open in $\partial\mathbb{R}^N_+$, then $d_n=\text{dist}(x_n,\partial\Omega)=\text{dist}(x_n,\Gamma)=x_{n,N}$ for all $n$ large enough. Arguing as in the previous case, we define
\[v_n(y)=\eta_n^{\frac{2}{p-1}}u_n(x_n+\eta_n y)\quad\forall y\in\Omega_n\cup\Gamma_n,\]
where $\eta_n=\|u_n\|^{-\frac{p-1}{2}}_{L^\infty(\Omega)},$ $0\in \Omega_n=B_{d/\eta_n}(0)\cap\left\{y_N>-d_n/\eta_n\right\}$ and $\Gamma_n=B_{d/\eta_n}(0)\cap\left\{y_N=-d_n/\eta_n\right\}$ for some $d>0$. It is not difficult to see that $v_n$ is well-defined for all $n$ large enough and it satisfies  
\begin{equation*}
\begin{cases}
\displaystyle -\div(M_n(y)\nabla v_n)+b_n(y)\nabla v_n+g_n(y,u_n)u_n\frac{M_n(y)\nabla v_n\nabla v_n}{v_n} &
\\
\displaystyle =\lambda v_n^p \frac{f(u_n)}{u_n^p}+\eta_n^{\frac{2(p-\sigma)}{p-1}} t_n v_n^{\sigma}, &y\in\Omega_n,
\\
v_n>0, &x\in\Omega_n,
\\
v_n=0, &y\in\Gamma_n,
\end{cases}
\end{equation*}
where $M_n(y)=M(x_n+\eta_n y)$, $b_n(y)=\eta_n b(x_n+\eta_n y)$ and $g_n(y,\cdot)=g(x_n+\eta_n y,\cdot)$. 

Now, if $\{d_n/\eta_n\}$ is unbounded, we can extract a subsequence such that $d_n/\eta_n\to+\infty$ as $n\to+\infty$. In this case, $\bigcup_{n\in\mathbb{N}}\Omega_n=\mathbb{R}^N$, so we can argue as in the case $x_0\in\Omega$ without relevant changes and arrive to a contradiction. 

Let us assume now that $\{d_n/\eta_n\}$ is bounded. Then, up to a (not relabeled) subsequence, $d_n/\eta_n\to \kappa$ for some $\kappa\geq 0$. Thus, $\bigcup_{n\in\mathbb{N}}\Omega_n=\{y_N>\kappa\}$. Let us prove the estimates in this new situation. 

\noindent{\bf Step 2.2) A priori estimates.}

Let us denote 
\[T_n=(0,...,0,d_n/\eta_n)\in\mathbb{R}^N,\quad \Omega_n'=B_{d/\eta_n}(T_n)\cap \mathbb{R}^N_+,\quad \Gamma_n'=B_{d/\eta_n}(T_n)\cap \partial\mathbb{R}^N_+.\] Observe that $\bigcup_{n\in\mathbb{N}}\Omega_n'=\mathbb{R}^N_+$ and $\bigcup_{n\in\mathbb{N}}\Gamma_n'=\partial\mathbb{R}^N_+$. Let us define $w_n:\Omega_n'\cup\Gamma_n'\to [0,+\infty)$ by 
\[w_n(y)=v_n(y-T_n)\quad\forall y\in\Omega_n'\cup\Gamma_n'.\]
Obviously, $w_n$ satisfies 
\begin{equation}
\label{translatedproblem}
\begin{cases}
\displaystyle -\div(M_n'(y)\nabla w_n)+ b_n'(y)\nabla w_n+g'_n\left(y,\frac{w_n}{\eta_n^\frac{2}{p-1}}\right)\frac{1}{\eta_n^\frac{2}{p-1}}M_n'(y)\nabla v_n\nabla v_n &
\\
\displaystyle =\lambda \eta_n^\frac{2p}{p-1} f\left(\frac{w_n}{\eta_n^\frac{2}{p-1}}\right)+\eta_n^{\frac{2(p-\sigma)}{p-1}} t_n w_n^{\sigma}, &y\in\Omega'_n,
\\
w_n>0, &x\in\Omega'_n,
\\
w_n=0, &y\in\Gamma_n',
\end{cases}
\end{equation}
where $M_n'(y)=M_n(y-T_n)$, $b_n'(y)=b_n(y-T_n)$ and $g_n'(y,\cdot)=g_n(y-T_n,\cdot)$. We also have that $\|w_n\|_{L^\infty(\Omega_n')}=w_n(T_n)=1$ for all $n$. 

Let us denote $T_\infty=(0,...,0,\kappa)\in\mathbb{R}^N$ and, for $R>\kappa$, let us consider the sets 
\[\omega=B_R(T_\infty)\cap\mathbb{R}^N_+,\quad \beta=B_R(T_\infty)\cap\partial\mathbb{R}^N_+\quad \omega_1=B_{2R}(T_\infty)\cap\mathbb{R}^N_+,\quad \beta_1=B_{2R}(T_\infty)\cap \partial\mathbb{R}^N_+.\]It is easy to see that, for all $n$ large enough, 
\begin{equation*}
T_n\in\omega\subset\omega_1\subset\Omega_n',\quad\quad0\in\beta\subset\beta_1\subset \Gamma_n',\quad\quad \overline{\omega}\subseteq\omega_1\cup\beta_1.
\end{equation*}
In particular, $w_n$ satisfies problem \eqref{translatedproblem} by changing $\Omega_n'$ with $\omega_1$ and $\Gamma_n'$ with $\beta_1$, and further, $\|w_n\|_{L^\infty(\omega_1)}=1$ for all $n$ large enough. Thus, it follows again from Lemma~\ref{holderestimatelemma} in the Appendix (thanks to conditions \eqref{sgsmallcond} and \eqref{boundsf} and to Lemma~\ref{summabilitylemma}) that there exist $C>0,\alpha\in(0,1)$ such that 
\[\|w_n\|_{C^{0,\alpha}(\overline{\omega})}\leq C\]
for every $n$ large enough. As a consequence, there exists $v\in C(\overline{\omega})$ such that, up to a subsequence, $w_n\to v$ uniformly in $\overline{\omega}$. In particular, $v=0$ on $\beta$.

On the other hand, from the H\"older estimate one can also deduce that $\kappa>0$. Indeed, 
\[1=|w_n(T_n)-w_n(0)|\leq C|T_n|^\alpha=C(d_n/\eta_n)^\alpha\to C\kappa^\alpha.\]

Next, an estimate on $\{w_n\}$ in $H^1_{\mbox{\tiny loc}}(\omega)$ can be proven as in CASE 1, so that $w_n\rightharpoonup v$ weakly in $H^1_{\mbox{\tiny loc}}(\omega)$, up to a subsequence. Moreover, the same arguments are valid to prove that $\{w_n\}$ satisfies also \eqref{lowerlocalestimate} and, furthermore, that $\nabla w_n\to\nabla v$ a.e. in $\omega$. 

\noindent{\bf Step 2.3) Passing to the limit.}

We can now pass to the limit as in CASE 1 and deduce that $v\in H^1_{\mbox{\tiny loc}}(\omega)\cap C(\overline{\omega})$ is a solution to the following problem:
\begin{equation*}
\begin{cases}
\displaystyle-\div(M(x_0)\nabla v)+\mu(x_0)\frac{M(x_0)\nabla v\nabla v}{v}=\lambda v^p,\quad &y\in \omega,
\\
v>0, &y\in \omega,
\\
v=0,\quad &y\in \beta.
\end{cases}
\end{equation*}
Actually, as $R$ is arbitrary, the diagonal argument (see \cite{GidasSpruck}) implies that $v\in H^1_{\mbox{\tiny loc}}(\mathbb{R}^N_+)\cap C\left(\overline{\mathbb{R}^N_+}\right)$ and it satisfies
\begin{equation*}
\begin{cases}
\displaystyle-\div(M(x_0)\nabla v)+\mu(x_0)\frac{M(x_0)\nabla v\nabla v}{v}=\lambda v^p,\quad &y\in \mathbb{R}^N_+,
\\
v>0, &y\in \mathbb{R}^N_+,
\\
v=0,\quad &y\in \partial\mathbb{R}^N_+.
\end{cases}
\end{equation*}

\noindent{\bf Step 2.4) Conclusion.}

Observe that, if we denote $M(x_0)=(m_{ij})$ for $i,j=1,...,N$, then the previous equation may be written as
\[\displaystyle\sum_{i,j=1}^N m_{ij}\frac{\partial^2 v}{\partial y_i\partial y_j}+\frac{\mu(x_0)}{v}\sum_{i,j=1}^N m_{ij}\frac{\partial v}{\partial y_i}\frac{\partial v}{\partial y_j}=\lambda v^p,\quad y\in\mathbb{R}^N_+.\]
Since $i,j$ commute in both $\frac{\partial^2 v}{\partial y_i\partial y_j}$ and $\frac{\partial v}{\partial y_i}\frac{\partial v}{\partial y_j}$, then a simple change of coordinates (see the conclusion of Case~1 in Section~2 of \cite{GidasSpruck} for the details) leads to finding a solution $w\in H^1_{\mbox{\tiny loc}}(\mathbb{R}^N_+)\cap C\left(\overline{\mathbb{R}^N_+}\right)$ to
\begin{equation*}
\begin{cases}
\displaystyle-\Delta w+\mu(y_0)\frac{|\nabla w|^2}{w}=w^p,\quad &y\in \mathbb{R}^N_+,
\\
w>0, &y\in \mathbb{R}^N_+,
\\
w=0,\quad &y\in \partial\mathbb{R}^N_+,
\end{cases}
\end{equation*}
for some $y_0\in\partial\Omega$. This contradicts Lemma~\ref{liouvillelemma} (see also Remark~\ref{liouvilleremark}). The proof is concluded.
\end{proof}

The following result, as Proposition~\ref{gidasspruck}, provides a priori  estimates on the solutions to \eqref{tproblem}. The difference lies on the fact that we do not impose the limit conditions at infinity \eqref{unifconv} and \eqref{limitf} at the expense of making a stronger restriction on $\sigma,p$.

\begin{proposition}\label{gidasspruck2}
Let $\lambda>0$, $f:[0,+\infty)\to[0,+\infty)$ be a continuous function and $g:\Omega\times (0,+\infty)\to [0,+\infty)$ be a Carath\'eodory function. For $\delta,\tau\geq 0$, $\sigma>0$ and $p\in\left(1,\frac{N+1}{N-1}\right)$, assume that $f$ satisfies \eqref{boundsf} and $g$ satisfies \eqref{sgsmallcond}. Assume in addition that $\sigma\leq\frac{N+1-(N-1)p}{2}$. Then, there exists $C>0$ such that $\|u\|_{L^\infty(\Omega)}\leq C$ for every solution $u$ to \eqref{tproblem} for all $t\in [0,t_0]$, where $t_0>0$ is given by Proposition~\ref{testimate}.
\end{proposition}

\begin{proof}
The proof is very similar to that of Proposition~\ref{gidasspruck}. Here we give only a sketch.

Arguing by contradiction, we assume that there exist two sequences $\{t_n\}\subset [0,t_0]$ and $\{u_n\}$ such that $u_n$ is a solution to $(P^{t_n})$ for all $n$ and $\|u_n\|_{L^\infty(\Omega)}\to+\infty$ as $n\to+\infty$. We also consider a sequence $\{x_n\}\subset\Omega$ satisfying
\[\|u_n\|_{L^\infty(\Omega)}=u_n(x_n)\quad\forall n,\quad x_n\to x_0\in\overline{\Omega}\text{, up to a subsequence};\]
this can be done by virtue of Lemma~\ref{continuouslemma}. We denote $\eta_n=\|u_n\|_{L^\infty(\Omega)}^{-\frac{p-1}{2}}$. Let us assume that $x_0\in\partial\Omega$ (we omit the simpler case $x_0\in\Omega$). Arguing as in the proof of Proposition~\ref{gidasspruck}, we assume without loss of generality that $u_n$ is a solution to $(R^{t_n})$ for all $n$, with $\Omega=V\subset\mathbb{R}^N_+$ and $\Gamma\subset\partial\Omega\cap\partial\mathbb{R}^N_+$, being $\Gamma$ open as a subset of $\partial\mathbb{R}^N_+$ and connected. Thus, there exists a sequence of bounded domains $\{\Omega_n\}$ satisfying, for every $n$, that $0\in\Omega_n$, $x_n+\eta_n y\in\Omega$ for all $y\in \Omega_n$ and $\bigcup_{n\in\mathbb{N}}\Omega_n=X$, where $X$ may be either $\mathbb{R}^N$ or $\{y_N>\kappa\}$ for some $\kappa\geq 0$.

In any case, we define $v_n:\Omega_n\to\mathbb{R}$ by
\[v_n(y)=\eta_n^{\frac{2}{p-1}}u_n(x_n+\eta_n y)\quad\forall y\in \Omega_n.\]
It is easy to check that $v_n$ satisfies the equation 
\begin{align*}
-\div(M_n(y)\nabla v_n)&+b_n(y)\nabla v_n+u_n g_n(y,u_n)\frac{M_n(y)\nabla v_n\nabla v_n}{v_n}
\\
&=\lambda v_n^p\frac{f(u_n)}{u_n^p}+\eta_n^{\frac{2(p-\sigma)}{p-1}} t_n v_n^{\sigma},\quad y\in\Omega_n,
\end{align*}
where $M_n(y)=M(x_n+\eta_n y)$, $b_n(y)=\eta_n b(x_n+\eta_n y)$ and $g_n(y,\cdot)=g(x_n+\eta_n y,\cdot)$. From the previous equation one obtains the same local estimates as in Proposition~\ref{gidasspruck}. However, in this case one cannot pass to the limit directly in the equation. We overcome this issue by deducing from $sg_n(x,s)\leq\sigma$, $M_n(y)\nabla v_n\nabla v_n\geq 0$ and $f(u_n)\geq u_n^p$ the inequality
\[-\div(M_n(y)\nabla v_n)+b_n(y)\nabla v_n+\sigma\frac{M_n(y)\nabla v_n\nabla v_n}{v_n}
\geq\lambda v_n^p,\quad y\in\Omega_n.\]
In fact, one can pass to the limit in the previous inequality as in Proposition~\ref{gidasspruck} (but using Fatou's lemma only once). Applying after that a convenient change of coordinates, we find a supersolution $v\in H^1_{\mbox{\tiny loc}}(X)\cap C(\overline{X})$ to \eqref{Xprob}, 
where either $X=\mathbb{R}^N$ or $X=\mathbb{R}^N_+$. This is a contradiction with Lemma~\ref{liouvillelemmasupersol}.
\end{proof}

Next proposition provides similar estimates as Proposition~\ref{gidasspruck} and Proposition~\ref{gidasspruck2}. The novelty is that $g$ satisfies now a limit condition at infinity only in a neighborhood of $\partial\Omega$. This means that we need to impose stronger restrictions on $\sigma,p$ than in Proposition~\ref{gidasspruck}, but milder than in Proposition~\ref{gidasspruck2}.

\begin{proposition}\label{gidasspruck3}
Let $\lambda>0$, $f:[0,+\infty)\to[0,+\infty)$ be a continuous function and $g:\Omega\times (0,+\infty)\to [0,+\infty)$ be a Carath\'eodory function. For $\delta,\tau\geq 0$, $\sigma>0$ and $p\in \left(1,\frac{N}{N-2}\right)$, assume that $f$ satisfies \eqref{boundsf} and \eqref{limitf}, and that $g$ satisfies \eqref{sgsmallcond} and \eqref{unifconvboundary}. Assume in addition that $\sigma\leq\frac{N-(N-2)p}{2}$. Then, there exists $C>0$ such that $\|u\|_{L^\infty(\Omega)}\leq C$ for every solution $u$ to \eqref{tproblem} for all $t\in [0,t_0]$, where $t_0>0$ is given by Proposition~\ref{testimate}.
\end{proposition}

\begin{proof}
We argue similarly as for Proposition~\ref{gidasspruck} and Proposition~\ref{gidasspruck2}, so we provide only the main ideas.

Assume by contradiction that there exist two sequences $\{t_n\}\subset [0,t_0]$ and $\{u_n\}$ such that $u_n$ is a solution to $(P^{t_n})$ for all $n$ and $\|u_n\|_{L^\infty(\Omega)}\to+\infty$ as $n\to+\infty$. Thanks to Lemma~\ref{continuouslemma}, we may also consider a sequence $\{x_n\}\subset\Omega$ satisfying
\[\|u_n\|_{L^\infty(\Omega)}=u_n(x_n)\quad\forall n,\quad x_n\to x_0\in\overline{\Omega}\text{ up to a subsequence}.\]
Suppose that $x_0\in\Omega$. Since $x_0$ might belong to $\omega$ (where we know nothing about the asymptotic behavior of $g$ at infinity), we cannot proceed as in the proof of Proposition~\ref{gidasspruck}. Nevertheless, scaling $u_n$ conveniently and using \eqref{sgsmallcond} and \eqref{boundsf} we may argue as in the proof of Proposition~\ref{gidasspruck2} to find a supersolution $0<v\in H^1_{\mbox{\tiny loc}}(\mathbb{R}^N)\cap  C(\mathbb{R}^N)$ to
\begin{equation*}
\displaystyle-\Delta v+\sigma\frac{|\nabla v|^2}{v}=v^p, \quad y\in \mathbb{R}^N.
\end{equation*}
This is a contradiction with Lemma~\ref{liouvillelemmasupersol}.

On the other hand, if $x_0\in\partial\Omega$, then we may take advantage of  \eqref{unifconvboundary} to obtain, using also \eqref{limitf} and arguing as in Proposition~\ref{gidasspruck}, a solution $v\in H^1_{\mbox{\tiny loc}}(X)\cap  C(\overline{X})$ to \eqref{Xprob} for $\sigma=\mu(y_0)$ and some $y_0\in\overline{\Omega}$, and either $X=\mathbb{R}^N$ or $X=\mathbb{R}^N_+$. This contradicts Lemma~\ref{liouvillelemma} (see also Remark~\ref{liouvilleremark}).
\end{proof}

Next result provides an estimate for the solutions to problem \eqref{lambdaproblem} whose dependence on $\lambda$ is explicit. As a consequence, it is shown that the norm of the solutions to problem \eqref{lambdaproblem}, if they exist, becomes arbitrarily small as $\lambda$ tends to infinity.

\begin{proposition}\label{gidasspruck4}
Let $f:[0,+\infty)\to[0,+\infty)$ be a continuous function and $g:\Omega\times (0,+\infty)\to [0,+\infty)$ be a Carath\'eodory function. For $\delta=0$, $\tau\geq 0$, $\sigma>0$ and $p\in\left(1,\frac{N+1}{N-1}\right)$, assume that $f$ satisfies \eqref{boundsf} and $g$ satisfies \eqref{sgsmallcond}. Assume in addition that $\sigma\leq \frac{N+1-(N-1)p}{2}$. Then, there exists $C>0$ such that 
\begin{equation*}
\lambda^\frac{1}{p-1}\|u\|_{L^\infty(\Omega)}\leq C
\end{equation*}
for every solution $u$ to \eqref{lambdaproblem} for all $\lambda>0$.
\end{proposition}

\begin{proof}
Arguing again as in the proof of Proposition~\ref{gidasspruck}, assume that there exist two sequences $\{\lambda_n\}\subset [0,+\infty)$ and $\{u_n\}$ such that $u_n$ is a solution to $(P_{\lambda_n})$ for all $n$ and $\|z_n\|_{L^\infty(\Omega)}\to+\infty$ as $n\to+\infty$, where $z_n=\lambda_n^\frac{1}{p-1}u_n$. It is easy to see that $z_n$ satisfies
\[-\Delta z_n+\lambda_n^{-\frac{1}{p-1}} g\left(x,\lambda_n^{-\frac{1}{p-1}}z_n\right)|\nabla z_n|^2=\lambda_n^\frac{p}{p-1}f\left(\lambda_n^{-\frac{1}{p-1}}z_n\right),\quad x\in\Omega.\]
Since $z_n\in C(\overline{\Omega})$ by virtue of Lemma~\ref{continuouslemma}, then we may take $\{x_n\}\subset\Omega$ such that $z_n(x_n)=\|z_n\|_{L^\infty(\Omega)}$ for all $n$. Let $x_0\in\overline{\Omega}$ be such that, passing to a subsequence if necessary, $x_n\to x_0\in\partial\Omega$ (the case $x_0\in\Omega$ is analogous, so we omit it). Arguing as in the proof of Proposition~\ref{gidasspruck} (Case~2), we may assume without loss of generality that $v_n\in H^1(\Omega_n)\cap L^\infty(\Omega_n)$, defined by $v_n(y)=\eta_n^\frac{2}{p-1}z_n(x_n+\eta_n y)$ for all $y\in\Omega_n$, satisfies 
\begin{equation}\label{Lnequation}
-\div(M_n(y)\nabla v_n)+b_n(y)\nabla v_n + \frac{v_n}{L_n} g_n\left(y,\frac{v_n}{L_n}\right)\frac{M_n(y)\nabla v_n\nabla v_n}{v_n}= L_n^p f\left(\frac{v_n}{L_n}\right),\quad y\in \Omega_n,
\end{equation}
where $\eta_n=\|z_n\|_{L^\infty(\Omega)}^{-\frac{p-1}{2}}$,  $L_n^{p-1}=\lambda_n\eta_n^2$ and $\Omega_n, \Gamma_n, M_n, b_n, g_n$ are as in the proof of Proposition~\ref{gidasspruck} (Case~2). From \eqref{Lnequation}, and using conditions \eqref{boundsf} and \eqref{sgsmallcond}, one can prove the same local estimates on $\{v_n\}$ as in the proof of Proposition~\ref{gidasspruck}. Moreover, again by \eqref{boundsf} and \eqref{sgsmallcond} we deduce that
\begin{equation*}
-\div(M_n(y)\nabla v_n)+b_n(y)\nabla v_n + \sigma\frac{M_n(y)\nabla v_n\nabla v_n}{v_n}\geq v_n^p,\quad y\in \Omega_n.
\end{equation*}
Now we pass to the limit in the previous inequality by using Fatou's lemma once and obtain a supersolution $v\in H^1_{\mbox{\tiny loc}}(X)\cap C(\overline{X})$ to \eqref{Xprob}, where either $X=\mathbb{R}^N$ or $X=\mathbb{R}^N_+$. This is a contradiction by virtue of Lemma~\ref{liouvillelemmasupersol}.
\end{proof}

In the last result of this section we impose stronger restrictions on $f$ and $g$ to obtain an a priori estimate for the solutions to \eqref{lambdaproblem} similar to that in Proposition~\ref{gidasspruck4}. The advantage is that $p$ can be chosen near $2^*-1$.

\begin{proposition}\label{gidasspruck5}
Let $g:\overline{\Omega}\times (0,+\infty)\to [0,+\infty)$ be a continuous function. For $\delta=0$, $\tau\geq 0$, $\sigma>0$ and $p\in\left(1,2^*-1\right)$, assume that $g$ satisfies \eqref{sgsmallcond}, \eqref{unifconv} and \eqref{sguniflimitzero}.  Assume in addition that there exists $\alpha>0$ such that, for each $x_0\in\overline{\Omega}$ and $L>0$, the function $\psi:[0,+\infty)\to [0,+\infty)$, defined by
\[\psi(s)=\int_0^s e^{-\frac{1}{L}\int_{\alpha L}^t g\left(x_0,\frac{r}{L}\right)dr}dt\quad\forall s\geq 0,\]
satisfies that 
\begin{equation}\label{Hdecreasing}
s\mapsto \frac{\psi'(s)s^p}{\psi(s)^{2^*-1}}\quad\text{ is decreasing for all }s>0.
\end{equation}
Lastly, let us consider the function $f:[0,+\infty)\to [0,+\infty)$ defined by $f(s)=s^p$ for all $s\geq 0$. Then, there exists $C>0$ such that 
\begin{equation*}
\lambda^\frac{1}{p-1}\|u\|_{L^\infty(\Omega)}\leq C
\end{equation*}
for every solution $u$ to \eqref{lambdaproblem} for all $\lambda>0$.
\end{proposition}

\begin{proof}
We may reproduce the proof of Proposition~\ref{gidasspruck4} up to \eqref{Lnequation}, where $\{v_n\}$ satisfies the same estimates as in the proof of Proposition~\ref{gidasspruck}. Now, unlike in the proof of Proposition~\ref{gidasspruck4}, we aim to pass to the limit directly in \eqref{Lnequation}. In order to do so, let us take a not relabeled subsequence so that $L_n\to L$ for some $L\in [0,+\infty]$. Let us assume first that $L=0$. In this case, we can pass to the limit in \eqref{Lnequation} as in the proof of Proposition~\ref{gidasspruck} and we obtain a solution $v\in H^1_{\mbox{\tiny loc}}(X)\cap C(\overline{X})$ to \eqref{Xprob} with $\sigma=\mu(y_0)$ for some $y_0\in\overline{\Omega}$, where either $X=\mathbb{R}^N$ or $X=\mathbb{R}^N_+$. This is a contradiction by virtue of Lemma~\ref{liouvillelemma} (see also Remark~\ref{liouvilleremark}).

Therefore, necessarily $L>0$. Assume now that $L=+\infty$. Using \eqref{sguniflimitzero}, we may pass to the limit again in \eqref{Lnequation} and we find a solution $v\in H^1_{\mbox{\tiny loc}}(X)\cap C(\overline{X})$ to \eqref{Xprob} with $\sigma=\mu_0(y_0)$ for some $y_0\in\overline{\Omega}$, where either $X=\mathbb{R}^N$ or $X=\mathbb{R}^N_+$. One more time, this contradicts Lemma~\ref{liouvillelemma} (see also Remark~\ref{liouvilleremark}).

The only remaining possibility is $0<L<+\infty$. Thanks to the continuity of $g$, we may pass to the limit in \eqref{Lnequation} once more to obtain a solution $v\in H^1_{\mbox{\tiny loc}}(X)\cap C(\overline{X})$ to \eqref{Xhprob}, where $h(s)=\frac{1}{L}g\left(y_0,\frac{s}{L}\right)$ for some $y_0\in\overline{\Omega}$ and either $X=\mathbb{R}^N$ or $X=\mathbb{R}^N_+$. Again, this is a contradiction with Lemma~\ref{liouvillelemma}.
\end{proof}


\section{Proofs of the main results}\label{existence}

We start by proving Theorem~\ref{comparison}. 

\begin{proof}[Proof of Theorem~\ref{comparison}]
Let us first define the function $\tilde{g}:\Omega\times(0,+\infty)\to [0,+\infty)$ by
\[\tilde{g}(x,s)=\frac{\sigma-s^\frac{1}{1-\sigma}g\left(x,s^\frac{1}{1-\sigma}\right)}{(1-\sigma)s},\quad\text{a.e. }x\in\Omega,\,\,\forall s>0.\]
Thanks to \eqref{sgsmallcond2} and \eqref{nondecrasingcond}, it is clear that $\tilde{g}$ is nonincreasing in the $s$ variable. Moreover, 
\[|\tilde{g}(s,x)|\leq\frac{\sigma+s^\frac{1}{1-\sigma} g_0\left(s^\frac{1}{1-\sigma}\right)}{(1-\sigma)s},\quad\text{a.e. }x\in\Omega,\,\,\forall s>0,\]
so it is a bounded function in the $x$ variable.
 
For some $0\leq \phi\in H_0^1(\Omega)\cap L^\infty(\Omega)$ with compact support, let us take $\frac{(1-\sigma)\phi}{u^\sigma}$ as test function in \eqref{compsubsol}. Then, if we denote $\tilde{u}=u^{1-\sigma}$, we deduce that
\[\int_\Omega\nabla\tilde{u}\nabla\phi\leq\int_\Omega \tilde{g}(x,\tilde{u})|\nabla \tilde{u}|^2\phi+\int_\Omega\frac{(1-\sigma)h(x)}{\tilde{u}^\frac{\sigma}{1-\sigma}}\phi.\]
Arguing similarly, $\tilde{v}=v^{1-\sigma}$ satisfies
\[\int_\Omega\nabla\tilde{v}\nabla\phi\geq\int_\Omega \tilde{g}(x,\tilde{v})|\nabla \tilde{v}|^2\phi+\int_\Omega\frac{(1-\sigma)h(x)}{\tilde{v}^\frac{\sigma}{1-\sigma}}\phi.\]
Moreover, $\tilde{u},\tilde{v}\in C(\Omega)\cap W^{1,N}_{\mbox{\tiny loc}}(\Omega)$ and they satisfy $\limsup_{x\to x_0}(\tilde{u}(x)-\tilde{v}(x))\leq 0$ for all $x_0\in\partial\Omega$. 

At this point one can reproduce the proof of \cite[Theorem 2.1]{L} without relevant changes and conclude that $\tilde{u}\leq\tilde{v}$ in $\Omega$. Equivalently, $u\leq v$ in $\Omega$.  
\end{proof}

Next we prove Theorem~\ref{nonexistence}.

\begin{proof}[Proof of Theorem~\ref{nonexistence}]
Let $v$ be a solution to \eqref{hproblem}. Recall that the definition of solution implies that there exists $\gamma>0$ such that $v^\gamma\in H_0^1(\Omega)$. It is easy to see that $u=v^\gamma$ satisfies the equation
\begin{equation}\label{hgammaproblem}
-\Delta u+g_\gamma(x,u)|\nabla u|^2=\gamma v^{\gamma-1}h(x),\quad x\in\Omega,
\end{equation}
where 
\[g_\gamma(x,s)=\frac{s^\frac{1}{\gamma}g\left(x,s^\frac{1}{\gamma}\right)+\gamma-1}{\gamma s},\quad\text{a.e. }x\in\Omega,\,\,\forall s>0.\]
It follows from \eqref{gtoolarge} that 
\begin{equation}\label{gammasgtoolarge}
s g_\gamma(x,s)\geq\frac{\tau+\gamma-1}{\gamma}\quad\text{a.e. }x\in\Omega\setminus\omega,\,\,\forall s\in (0,s_0^\gamma),
\end{equation}
where clearly $\frac{\tau+\gamma-1}{\gamma}>1$. From now we will denote $\rho=\frac{\tau+\gamma-1}{\gamma}$ and $s_1=s_0^\gamma$.

Recall that there exists $c>0$ such that $v\geq c^\frac{1}{\gamma}$ in $\omega$. For every $\varepsilon\in (0,\min\{s_1,c\})$, let us define the function $\varphi_\varepsilon:[0,+\infty)\to[0,+\infty)$ by
\[\varphi_\varepsilon(s)=
\begin{cases}
\displaystyle \frac{1}{s}+\frac{s^{\rho-1}-\varepsilon^{\rho-1}}{(\rho-1)s^\rho}, &s\geq\varepsilon,
\\
\displaystyle\frac{s}{\varepsilon^2}, &0\leq s<\varepsilon.
\end{cases}
\]
Clearly, $\varphi_\varepsilon(u)\in H_0^1(\Omega)\cap L^\infty(\Omega)$. Hence, taking $\varphi_\varepsilon(u)$ as test function in the weak formulation of \eqref{hgammaproblem} (this can be done thanks to Remark~\ref{testfunctionsremark} because $u\in H_0^1(\Omega)$) we obtain
\begin{equation}\label{ecu0}
\int_\Omega\varphi_\varepsilon'(u)|\nabla u|^2+\int_\Omega g_\gamma(x,u)|\nabla u|^2\varphi_\varepsilon(u)=\gamma\int_\Omega v^{\gamma-1}h(x)\varphi_\varepsilon(u).
\end{equation}
Observe that, from \eqref{gammasgtoolarge}, it follows that
\begin{align}\label{ecu1}
\nonumber\int_\Omega &g_\gamma(x,u)|\nabla u|^2 \varphi_\varepsilon(u) =\int_{\omega\cup\{u\geq s_1\}} g_\gamma(x,u)|\nabla u|^2 \varphi_\varepsilon(u)+\int_{(\Omega\setminus\omega)\cap\{u<s_1\}} g_\gamma(x,u)|\nabla u|^2 \varphi_\varepsilon(u)
\\
&\geq-C+ \int_{(\Omega\setminus\omega)\cap \{\varepsilon<u<s_1\}} \frac{\rho|\nabla u|^2}{u}\left(\frac{1}{u}+\frac{u^{\rho-1}-\varepsilon^{\rho-1}}{(\rho-1)u^\rho}\right).
\end{align}
On the other hand,
\begin{equation}\label{ecu2}
\int_\Omega\varphi_\varepsilon'(u)|\nabla u|^2=\int_{\{u>\varepsilon\}}\frac{\rho|\nabla u|^2}{u^2}-\int_{\{u>\varepsilon\}}\frac{\rho|\nabla u|^2}{u}\left(\frac{1}{u}+\frac{u^{\rho-1}-\varepsilon^{\rho-1}}{(\rho-1)u^\rho}\right)+\int_{\{u\leq\varepsilon\}}\frac{|\nabla u|^2}{\varepsilon^2}.
\end{equation}
We will now absorb the negative term in \eqref{ecu2} with the last term in \eqref{ecu1}. Indeed,
\begin{align*}
\int_{(\Omega\setminus\omega)\cap \{\varepsilon<u<s_1\}} &\frac{\rho|\nabla u|^2}{u}\left(\frac{1}{u}+\frac{u^{\rho-1}-\varepsilon^{\rho-1}}{(\rho-1)u^\rho}\right)-\int_{\{u>\varepsilon\}}\frac{\rho|\nabla u|^2}{u}\left(\frac{1}{u}+\frac{u^{\rho-1}-\varepsilon^{\rho-1}}{(\rho-1)u^\rho}\right)
\\
&=-\int_{[\{\varepsilon<u<s_1\}\cap\omega]\cup\{u\geq s_1\}}\frac{\rho|\nabla u|^2}{u}\left(\frac{1}{u}+\frac{u^{\rho-1}-\varepsilon^{\rho-1}}{(\rho-1)u^\rho}\right)
\\
&\geq -\int_{\{u\geq \min\{s_1,c\}\}}\frac{\rho|\nabla u|^2}{u}\left(\frac{1}{u}+\frac{u^{\rho-1}-\varepsilon^{\rho-1}}{(\rho-1)u^\rho}\right)\geq -C.
\end{align*}
In conclusion, from \eqref{ecu0}, \eqref{ecu2} and from the previous discussion we deduce that
\[\int_{\{u>\varepsilon\}}\frac{\rho|\nabla u|^2}{u^2}\leq \gamma\int_\Omega v^{\gamma-1}h(x)\varphi_\varepsilon(u)+C.\]
Since $\varphi_\varepsilon(s)\leq \frac{\rho}{(\rho-1)s}$ for all $s>0$, we finally deduce that
\[\int_{\{u>\varepsilon\}}\frac{\rho|\nabla u|^2}{u^2}\leq C\left(\int_\Omega\frac{v^{\gamma-1}|h(x)|}{u}+1\right)=C\left(\int_\Omega\frac{|h(x)|}{v}+1\right).\]
Therefore, we let $\varepsilon$ tend to zero and by virtue of Fatou's lemma we obtain that
\[\int_\Omega\frac{|\nabla u|^2}{u^2}\leq C\left(\int_\Omega\frac{|h(x)|}{v}+1\right).\]
Now, in \cite{ZWQ} it is proven that $\int_\Omega\frac{|\nabla u|^2}{u^2}=+\infty$. Therefore, $\int_\Omega\frac{|h(x)|}{v}=+\infty$ and the proof is finished.
\end{proof}

We are ready now to prove Theorem~\ref{fxmodelthm}.

\begin{proof}[Proof of Theorem~\ref{fxmodelthm}]
The existence of a solution $u\in H_0^1(\Omega)$ (not necessarily bounded) is a consequence of \cite[Theorem 3.1]{MA}. Moreover, since $q>\frac{N}{2}$, the well-known Stampacchia's lemma (see \cite{Stampacchia}) implies that $u\in L^\infty(\Omega)$. 

On the other hand, Lemma~\ref{continuouslemma} leads to $u\in C(\overline{\Omega})$. Furthermore, $u\in W^{1,N}_{\mbox{\tiny loc}}(\Omega)$ by virtue of Lemma~\ref{reglemma} and Remark \ref{strongsolremark} in the Appendix below. Therefore, the uniqueness of solution follows directly from Theorem~\ref{comparison}. 

Finally, it is clear that, if $h\gneq 0$ has compact support, then $\int_\Omega\frac{h(x)}{u}<+\infty$ for every solution $u$ to \eqref{modelfxprob}. Thus, Theorem~\ref{nonexistence} implies the nonexistence part of the theorem.
\end{proof}

Now we prove Theorem~\ref{mainthm1}. Since the proofs of Theorem~\ref{mainthm2} and Theorem~\ref{mainthm3} are analogous, we will only make some comments below.

\begin{proof}[Proof of Theorem~\ref{mainthm1}]
We divide the proof into several steps.

\noindent{\bf Step 1) An auxiliary problem}

In addition to the hypotheses of Theorem~\ref{mainthm1}, let us further assume for now that $p<2^*-1$ and $g\gneq 0$. At the end of the proof of the theorem we will make the extension to general $p,g$. 

For every $0\lneq v\in C(\overline{\Omega})$, $t\geq 0$, $\lambda>0$, we consider the following problem:
\begin{equation}\label{auxprob}
\begin{cases}
\displaystyle-\Delta u+(v+\delta)g(x,v)\frac{|\nabla u|^2}{u+\delta}=\lambda f(v)+tv^\sigma,\quad & x\in\Omega,
\\
u>0,\quad & x\in\Omega,
\\
u=0,\quad & x\in\partial\Omega,
\end{cases}
\tag{$Q^t$}
\end{equation}
where $\sigma$ is the same real number that appears in condition \eqref{sgsmallcond}; note that $\sigma\in(0,1)$ since $g\gneq 0$. It is clear that, if $u$ is a solution to \eqref{auxprob} with $t=0$ and $v=u$, then it is actually a solution to \eqref{lambdaproblem}. In this step we will show that, if $v\gneq 0$, there exists a unique finite energy solution to \eqref{auxprob}. 

Observe that, if $0\lneq v\in C(\overline{\Omega})$, then $0\lneq\lambda f(v)+t v^\sigma\in C(\overline{\Omega})$ too. Therefore, if $\delta=0$, then Theorem~\ref{fxmodelthm} implies directly that there exists a unique finite energy solution to \eqref{auxprob}, that we will denote as $u_0$. On the other hand, let us assume that $\delta>0$. Then, $u_0$ is clearly a subsolution to \eqref{auxprob}. Let us consider the unique solution $w\in H_0^1(\Omega)\cap C(\overline{\Omega})$ to the linear problem
\[
\begin{cases}
-\Delta w=\lambda f(v)+tv^\sigma, &x\in\Omega,
\\
w>0, &x\in\Omega,
\\
w=0, &x\in\partial\Omega.
\end{cases}
\]
Then, $w$ is a supersolution to \eqref{auxprob}. Furthermore, $u_0\leq w$ in $\Omega$ by comparison, so \cite[Th\'eor\`eme~{3.1}]{BMP3} implies that there exists a finite energy solution $u_\delta$ to \eqref{auxprob} such that $0<u_0\leq u_\delta\leq w$ in $\Omega$. Moreover, every solution to \eqref{auxprob} belongs to $C(\overline{\Omega})\cap W^{1,N}_{\mbox{\tiny loc}}(\Omega)$ by virtue of Lemma~\ref{continuouslemma}, Lemma~\ref{reglemma} and Remark~\ref{strongsolremark}. Also, the function $(x,s)\mapsto \frac{(v(x)+\delta)g(x,v(x))}{s+\delta}$ satisfies \eqref{sgsmallcond2} and \eqref{nondecrasingcond}. Therefore, Theorem~\ref{comparison} implies that $u_\delta$ is the unique solution to \eqref{auxprob}.

\noindent{\bf Step 2) Existence of a finite energy solution to \eqref{lambdaproblem} if $p<2^*-1$ and $g\gneq 0$.}

Let $X=\{w\in C(\overline{\Omega}):w(x)\geq 0\,\,\forall x\in\overline{\Omega}\}$. We have shown in the previous step that the operator $K: X\times [0,+\infty)\to X$ given by 
\[
\begin{cases}
K(v,t)=u, \quad\forall v\in X\setminus\{0\},\,\, t\geq 0\text{, where $u$ is the unique solution to \eqref{auxprob}},
\\
K(0,t)=0,\quad\forall t\geq 0. 
\end{cases}
\]
is well-defined. We aim to prove that there exists $0\lneq u\in X$ such that $K(u,0)=u$, which is equivalent to finding a solution to \eqref{lambdaproblem}. 

We first prove that $K$ is continuous. Indeed, let $\{v_n\}\subset X$ and $\{t_n\}\subset [0,+\infty)$ be such that $v_n\to v$ in $C(\overline{\Omega})$ for some $v\in X$ and $t_n\to t$ for some $t\geq 0$. Let us denote $u_n=K(v_n,t_n)$ and $h_n=\lambda f(v_n)+t_nv_n^\sigma$. Taking into account that $g\geq 0$, it is clear that
\begin{equation}\label{hnineq}
-\Delta u_n\leq h_n,\quad x\in \Omega,\,\,\forall n.
\end{equation}
We know that there exists $C>0$ such that $h_n\leq C$ for all $n$. Then, \eqref{hnineq} leads to $-\Delta u_n\leq C$ in $\Omega$. From this inequality, it is straightforward to prove that $\{u_n\}$ is bounded in $H_0^1(\Omega)$. Hence, up to a subsequence, $u_n\rightharpoonup u$ weakly in $H_0^1(\Omega)$ and $u_n\to u$ strongly in $L^m(\Omega)$ for some $0\leq u\in H_0^1(\Omega)$ and for all $m\in[1,2^*)$. Furthermore, the inequality \eqref{hnineq} yields the following estimate by virtue of the well-known Stampacchia's lemma (see \cite{Stampacchia}):
\begin{equation}\label{stampacchiaestimate}
\|u_n\|_{L^\infty(\Omega)}\leq C\|h_n\|_{L^\infty(\Omega)},\quad\forall n.
\end{equation}
In particular, if $v\equiv 0$, \eqref{stampacchiaestimate} yields directly that $u_n\to 0$ strongly in $C(\overline{\Omega})$. That is to say, $K$ is continuous in $\{0\}\times [0,+\infty)$.


Let us assume on the contrary that $v\gneq 0$. In this case, we will prove following \cite{MA} that $\{u_n\}$ is locally bounded away from zero. Indeed, first of all observe that, using that $h_n\leq C$, we deduce from \eqref{stampacchiaestimate} that $\{u_n\}$ is bounded in $L^\infty(\Omega)$. Hence, using \eqref{boundsf}, \eqref{sgsmallcond} and the $L^\infty$ estimate on $u_n$, one can easily check that the function $w_n=u_n^{1-\sigma}$ satisfies
\[-\Delta w_n\geq \frac{(1-\sigma)\lambda v_n^p}{u_n^\sigma}\geq C v_n^p,\quad x\in\Omega,\]
for some constant $C>0$. On the other hand, let $z_n\in H_0^1(\Omega)\cap C(\overline{\Omega})$ be the unique solution to the linear problem
\[
\begin{cases}
-\Delta z_n= Cv_n^p, &x\in\Omega,
\\
z_n=0, &x\in\partial\Omega.
\end{cases}
\]
By comparison, $w_n\geq z_n$ for all $n$. Moreover, it is clear that $\{z_n\}$ is bounded in $H_0^1(\Omega)$ and, following \cite{MA} (which in turn is based on \cite{LU}), we also have that $\{z_n\}$ is bounded in $C^{0,\alpha}(\overline{\Omega})$ for some $\alpha\in (0,1)$. Hence, passing to a subsequence, $z_n$ converges weakly in $H_0^1(\Omega)$ and strongly in $C(\overline{\Omega})$ to the unique solution $z\in H_0^1(\Omega)\cap C(\overline{\Omega})$ to
\[
\begin{cases}
-\Delta z= Cv^p, &x\in\Omega,
\\
z=0, &x\in\partial\Omega.
\end{cases}
\]
Since $v\gneq 0$, the strong maximum principle implies that, for all $\omega\subset\subset\Omega$, there exists $c_\omega>0$ such that $z\geq 2 c_\omega$ in $\omega$. Furthermore, since $z_n\to z$ uniformly in $\overline{\Omega}$, it follows that $z_n\geq z-c_\omega$ in $\overline{\Omega}$ for all $n$ large enough. In sum,
\[u_n^{1-\sigma}=w_n\geq z_n\geq c_\omega,\quad x\in\omega\]
for all $n$ large enough, as we claimed.

It is clear now that $\{\Delta u_n\}$ is bounded in $L^1_{\mbox{\tiny loc}}(\Omega)$. Hence, from \cite{BoccMur} we deduce that $\nabla u_n\to\nabla u$ a.e. in $\Omega$. We may now pass to the limit using Fatou's lemma twice, similarly as in the proof of Proposition~\ref{gidasspruck} (Case 1, Step 1.3), so that $u$ is the unique solution to \eqref{auxprob} i.e. $u=K(v,t)$. 

It remains to prove that $u_n\to u$ strongly in $C(\overline{\Omega})$. Indeed, since $\{u_n\}$ is locally bounded away from zero, Lemma~\ref{reglemma}  implies that, for every $\omega\subset\subset\Omega$, one may take a subsequence such that $u_n\to u$ strongly in $C(\overline{\omega})$. Actually, the uniqueness of $u$ assures that the original sequence $\{u_n\}$, and not merely a subsequence, converges itself to $u$ strongly in $C(\overline{\omega})$. On the other hand, let $\zeta_n\in H_0^1(\Omega)\cap C(\overline{\Omega})$ be the unique solution to
\[
\begin{cases}
-\Delta \zeta_n= h_n, &x\in\Omega,
\\
\zeta_n=0, &x\in\partial\Omega.
\end{cases}
\]
By comparison, $u_n\leq \zeta_n$ in $\Omega$ for all $n$. Moreover, $\zeta_n\to\zeta$ strongly in $C(\overline{\Omega})$, where $\zeta\in H_0^1(\Omega)\cap C(\overline{\Omega})$ is the unique solution to 
\[
\begin{cases}
-\Delta \zeta= \lambda f(v)+tv^\sigma, &x\in\Omega,
\\
\zeta=0, &x\in\partial\Omega.
\end{cases}
\]
In particular, for fixed $\varepsilon>0$, we deduce that
\[
|u_n(x)-u(x)|\leq\zeta_n(x)+u(x)< \zeta(x)+u(x)+\frac{\varepsilon}{2}\quad\forall x\in\overline{\Omega},\,\,\forall n\text{ large enough.}
\]
Therefore, since $u,\zeta\in C(\overline{\Omega})$ with $u=\zeta=0$ on $\partial\Omega$, there exits $\omega\subset\subset\Omega$ such that
\[|u_n(x)-u(x)|<\varepsilon\quad\forall x\in\overline{\Omega}\setminus\omega,\,\,\forall n\text{ large enough.}\]
Since $u_n\to u$ strongly in $C(\overline{\omega})$, we conclude that
\[|u_n(x)-u(x)|<\varepsilon\quad\forall x\in\overline{\Omega},\,\,\forall n\text{ large enough.}\]
In other words, $u_n\to u$ strongly in $C(\overline{\Omega})$, so $K$ is continuous.

We prove now that $K$ is compact, i.e. it maps bounded sets to relatively compact sets. Indeed, let $\{v_n\}\subset X$ and $\{t_n\}\subset [0,+\infty)$ be bounded sequences. Taking subsequences, $v_n\to v$ in the weak-$\star$ topology of $L^\infty(\Omega)$ for some $v\in L^\infty(\Omega)$, while $t_n\to t$ for some $t\geq 0$. This is enough to pass to the limit in the equations as above. We conclude that, up to subsequences, $K(v_n,t_n)\to K(v,t)$ strongly in $C(\overline{\Omega})$. This proves that $K$ is compact. 

We will prove next that there exist $0<r<R$ and $t_1\geq 0$ such that
\begin{enumerate}
\item \label{rsphere}$u\not= s K(u,0)\,\,\forall s\in[0,1],\,\,\forall u\in X$ with $\|u\|_{L^\infty(\Omega)}=r$,
\
\item \label{Rsphere}$u\not= K(u,t)\,\,\forall t\geq 0,\,\,\forall u\in X$ with $\|u\|_{L^\infty(\Omega)}=R$,
\
\item \label{Rball}$u\not= K(u,t)\,\,\forall t\geq t_1,\,\,\forall u\in X$ with $\|u\|_{L^\infty(\Omega)}\leq R$.
\end{enumerate}

In order to prove item \eqref{rsphere}, let us assume by contradiction that, for all $r>0$, there exist $s\in [0,1]$ and $u\in X$ with $\|u\|_{L^\infty(\Omega)}=r$ such that $u= s K(u,0)$. In particular, $u>0$ in $\Omega$ and
\begin{equation*}
\begin{cases}
-\Delta u\leq s\lambda f(u),  &x\in\Omega,
\\
u=0,\ &x\in\partial\Omega.
\end{cases}
\end{equation*}
Since \eqref{fzero} holds, we can choose $r>0$ such that $\lambda f(t)\leq \frac{\lambda_1}{2}t$ for all $t\in [0,r]$, where $\lambda_1$ stands for the principal eigenvalue
of $-\Delta$ in $\Omega$ with zero Dirichlet boundary conditions. Furthermore, $u\leq r$ in $\Omega$, so we have that
\begin{equation*}
\begin{cases}
\displaystyle-\Delta u\leq\frac{\lambda_1}{2}u, &x\in\Omega,
\\
u>0, &x\in\Omega,
\\
u=0, &x\in\partial\Omega.
\end{cases}
\end{equation*}
This is a clear contradiction with the definition of $\lambda_1$.

On the other hand, if we take $R>C$, where $\|u\|_{L^\infty(\Omega)}\leq C$ (see Proposition~\ref{gidasspruck}) then it is clear that \eqref{Rsphere} holds. Moreover, if we take $t_1>t_0$, where $t_0>0$ is given by Proposition~\ref{testimate}, then \eqref{Rball} also holds.

In conclusion, \cite[Proposition~{2.1} and Remark~{2.1}]{deFigLiNuss} can be applied and in consequence we obtain a positive fixed point of $K_0$, i.e. a finite energy solution to \eqref{lambdaproblem}.

\noindent{\bf Step 3) Extension to the general case} 

From this point we assume only the hypotheses of Theorem~\ref{mainthm1}, i.e. $g$ need not be non-negative and $p$ might be greater than or equal to $2^*-1$. 

Fix $\lambda>0$. For some $\gamma>1$ that will be chosen later, let us consider the following problem:

\begin{equation}\label{equivalentproblem}
\begin{cases}
\dys-\Delta v+ g_\gamma(x,v)|\nabla v|^2= \lambda f_\gamma(v), &x\in \Omega,
\\
v> 0, &x\in \Omega,
\\
v=0, &x\in\partial \Omega,
\end{cases}
\end{equation}
where 
\begin{align*}
f_\gamma(s)&=b (s+\delta^\gamma)^\frac{\gamma-1}{\gamma}f\left((s+\delta^\gamma)^\frac{1}{\gamma}-\delta\right), \quad\forall s\geq 0,
\\
g_\gamma(x,s)&=\frac{(s+\delta^\gamma)^\frac{1}{\gamma} g\left(x,(s+\delta^\gamma)^\frac{1}{\gamma}-\delta\right)+\gamma-1}{\gamma (s+\delta^\gamma)},\quad \text{a.e. }x\in \Omega,\,\,\forall s>0,
\end{align*}
and $b>0$ is a constant, which depends only on $p,\delta,\gamma$, that will be chosen below too. It is easy to see that, if $v$ is a solution to \eqref{equivalentproblem}, then $u=(v+\delta^\gamma)^\frac{1}{\gamma}-\delta$ is a solution to \eqref{lambdaproblem}. Moreover, if $v\in H_0^1(\Omega)$ and $\delta>0$, then $u\in H_0^1(\Omega)$ too, so it has finite energy. We will now choose $\gamma>1$ so that problem \eqref{equivalentproblem} admits a finite energy solution.

Indeed, let us first denote $p_\gamma=\frac{\gamma-1+p}{\gamma}$. Obviously, $p_\gamma>1$. It is easy to see that the function 
\[s\mapsto \frac{s^{p_\gamma}}{\left(s+\delta^\gamma\right)^\frac{\gamma-1}{\gamma}\left[\left(s+\delta^\gamma\right)^\frac{1}{\gamma}-\delta\right]^p}\]
is bounded in $[0,+\infty)$. Then, using that $f$ satisfies \eqref{boundsf}, $b>0$ can be chosen (depending only on $p,\delta,\gamma$) such that
\[f_\gamma(s)\geq s^{p_\gamma}\quad\forall s\geq 0.\]
Moreover, since $f$ satisfies \eqref{boundsf}, it is clear that 
\[f_\gamma(s)\leq ab\gamma (s+\delta^\gamma)^{p_\gamma}\quad\forall s\geq 0.\]
We have proved that $f_\gamma$ satisfies \eqref{boundsf}. Taking into account that $f$ satisfies \eqref{limitf} and \eqref{fzero}, it is straightforward to check that also $f_\gamma$ satisfies \eqref{limitf} and \eqref{fzero} respectively.

On the other hand, since $g$ satisfies \eqref{unifconv}, then
\[\lim_{s\to+\infty}\left\|sg_\gamma(\cdot,s)-\frac{\mu+\gamma-1}{\gamma}\right\|_{L^\infty(\Omega)}=0.\]
Using that $\max_{x\in\overline{\Omega}}\mu(x)<\frac{2^*-1-p}{2^*-2}$, it is easy to see that 
\[\max_{x\in\overline{\Omega}}\frac{\mu(x)+\gamma-1}{\gamma}<\frac{2^*-1-p_\gamma}{2^*-2}.\]
Thus, $g_\gamma$ satisfies \eqref{unifconv}. It is also immediate to check that $g_\gamma$ satisfies \eqref{sgzerolimit} using that $g$ does too. Notice finally that, since $g$ satisfies \eqref{sgsmallcond}, then
\[\frac{\tau+\gamma-1}{\gamma}\leq \left(s+\delta^\gamma\right)g_\gamma(s,x)\leq \frac{\sigma+\gamma-1}{\gamma}\quad\text{a.e. }x\in\Omega,\,\,\forall s>0.\]
It is clear that 
\[2\frac{\sigma+\gamma-1}{\gamma}-1<\frac{\tau+\gamma-1}{\gamma}\leq \frac{\sigma+\gamma-1}{\gamma}<1.\]
Therefore, $g_\gamma$ satisfies \eqref{sgsmallcond}.

If we now choose $\gamma>1$ large enough so that $\frac{\tau+\gamma-1}{\gamma}>0$ and $p_\gamma<2^*-1$, then we may apply the previous step in order to obtain a finite energy solution to \eqref{equivalentproblem}. The proof is now finished.
\end{proof}

\begin{proof}[Proofs of Theorem~\ref{mainthm2} and Theorem~\ref{mainthm3}]
The proofs are analogous to that for Theorem~\ref{mainthm1}. The only essential difference lies in how the estimates on the solutions are obtained. In fact, instead of Proposition~\ref{gidasspruck}, we have to apply Proposition~\ref{gidasspruck3} for Theorem~\ref{mainthm2} and both Proposition~\ref{gidasspruck2} and Proposition~\ref{gidasspruck4} for Theorem~\ref{mainthm3}.
\end{proof}

Here we carry out the proof of Theorem~\ref{mainthm4}.

\begin{proof}[Proof of Theorem~\ref{mainthm4}]
Let us define the Carath\'eodory function $\bar{g}:\Omega\times (0,+\infty)\to\mathbb{R}$ by
\[\bar{g}(x,s)=
\begin{cases}
g(x,s)\quad &\text{if }x\in\Omega,\,\,s\in (0,s_0),
\\
\displaystyle\frac{s_0 g(x,s_0)}{s}\quad &\text{if }x\in\Omega,\,\,s\geq s_0.
\end{cases}
\]
Let us also consider the continuous function $\bar{f}:[0,+\infty)\to[0,+\infty)$ defined by
\[\bar{f}(s)=
\begin{cases}
f(s)\quad &\text{if }s\in (0,s_0),
\\
\displaystyle \frac{f(s_0)}{s_0^p}s^p&\text{if }s\geq s_0.
\end{cases}
\]
It is clear that $\bar{f}$ and $\bar{g}$ satisfy the hypotheses of Theorem~\ref{mainthm3}. Thus, if we write $(\bar{P}_\lambda)$ for denoting  problem \eqref{lambdaproblem} with $\bar{g}$ and $\bar{f}$ instead of $g$ and $f$, then Theorem~\ref{mainthm3} implies that there exists a solution $u_\lambda$ to $(\bar{P}_\lambda)$ for all $\lambda>0$ that satisfies
\begin{equation*}
\lambda^\frac{1}{p-1}\|u_\lambda\|_{L^\infty(\Omega)}\leq C\quad\forall\lambda>0.
\end{equation*}
In particular, there exists $\lambda_0>0$ such that $\|u_\lambda\|_{L^\infty(\Omega)}<s_0$ for every $\lambda\geq\lambda_0$. Hence, $u_\lambda$ is, in fact, a solution to \eqref{lambdaproblem} for every $\lambda\geq\lambda_0$. The proof is finished.
\end{proof}

It is left to prove Theorem~\ref{mainthm5}.

\begin{proof}[Proof of Theorem~\ref{mainthm5}]
Let $\sigma\in\left(0,\min\left\{\frac{1}{2},\frac{2^*-1-p}{2^*-2}\right\}\right)$. By virtue of Remark~\ref{muzeroremark}, we may take $\delta\in (0,1)$ such that
\[sg(x,s)\leq\sigma\quad\forall (x,s)\in\overline{\Omega}\times (0,\delta].\]
Notice that $\delta$ can be chosen as small as necessary. In fact, it will be taken small enough several times throughout the proof, depending its size only on $N,p,G$.

Let us define the function $\bar{g}:\overline{\Omega}\times (0,+\infty)\to [0,+\infty)$ by
\[\bar{g}(x,s)=
\begin{cases}
g(x,s)\quad &\text{if }(x,s)\in\overline{\Omega}\times (0,\delta],
\\
\displaystyle\frac{\delta g(x,\delta)}{s}\quad &\text{if }(x,s)\in\overline{\Omega}\times (\delta,+\infty).
\end{cases}
\]

It is clear that $\bar{g}$ satisfies \eqref{sgsmallcond}, \eqref{unifconv} and \eqref{sgzerolimit}. Thus, if we write $(\bar{P}_\lambda)$ for denoting problem \eqref{lambdaproblem} with $\bar{g}$ instead of $g$, then Theorem~\ref{mainthm1} implies that there exists a solution to $(\bar{P}_\lambda)$ for all $\lambda>0$. We aim now to prove that the hypotheses of Proposition~\ref{gidasspruck5} hold true, so that there exists $C>0$ such that 
\begin{equation}\label{lambdaestimate}
\lambda^\frac{1}{p-1}\|u\|_{L^\infty(\Omega)}\leq C\quad\forall u \text{ solution to }(\bar{P}_\lambda)\text{ with }\lambda>0.
\end{equation}
Observe that, if this is true, in particular there exists $\lambda_0>0$ such that $\|u\|_{L^\infty(\Omega)}\leq \delta$ for every solution $u$ to $(\bar{P}_\lambda)$ with $\lambda\geq\lambda_0$. Hence, the solutions to $(\bar{P}_\lambda)$ with $\lambda\geq\lambda_0$ are, in fact, solution to \eqref{lambdaproblem}. Thus, the proof will finish as soon as we prove \eqref{lambdaestimate}.

We will now prove that \eqref{Hdecreasing} in Proposition~\ref{gidasspruck5} holds. Indeed, for any $L>0$, $x_0\in\overline{\Omega}$, let us consider the function $\psi:[0,+\infty)\to[0,+\infty)$ given by
\[\psi(s)=\int_0^s e^{-\frac{1}{L}\int_{\delta L}^t \bar{g}\left(x_0,\frac{r}{L}\right)dr}dt\quad\forall s\geq 0.\]
It is straightforward to check that
\[\psi(s)=L\int_0^{\frac{s}{L}} e^{-\int_{\delta}^t \bar{g}\left(x_0,r\right)dr}dt\quad\forall s\geq 0.\]
We will prove next that $s\mapsto \frac{\psi'(s)s^p}{\psi(s)^{2^*-1}}$ is decreasing for all $s>0$, or equivalently, that
\begin{equation}\label{increasingineq}
(p-s\bar{g}(x_0,s))\frac{\psi(sL)}{L}< (2^*-1)s\psi'(sL)\quad\forall s>0.
\end{equation}
Notice that neither $\frac{\psi(sL)}{L}$ nor $\psi'(sL)$ depends on $L$.

In order to prove \eqref{increasingineq}, we distinguish two possibilities: either $0<s\leq\delta$ or $s>\delta$. Assume first that $0<s\leq\delta$. Then, 
\begin{align*}
&(p-s\bar{g}(x_0,s))\frac{\psi(sL)}{L}-(2^*-1)s\psi'(sL)
\\
=&(p-sg(x_0,s))\int_0^s e^{-\int_\delta^t g(x_0,r)dr}dt-(2^*-1)se^{-\int_\delta^s g(x_0,r)dr}.
\end{align*}
It follows from the mean value theorem that there exists $\theta\in (0,1)$ such that 
\begin{align*}
&(p-sg(x_0,s))\int_0^s e^{-\int_\delta^t g(x_0,r)dr}dt-(2^*-1)se^{-\int_\delta^s g(x_0,r)dr}
\\
=&s e^{\int_s^\delta g(x_0,r)dr}\left[ (p-sg(x_0,s)) e^{\int_{\theta s}^s g(x_0,r)dr}-(2^*-1)\right].
\end{align*}
Thanks to \eqref{integrablecond} we derive
\begin{equation*}
(p-sg(x_0,s)) e^{\int_{\theta s}^s g(x_0,r)dr}-(2^*-1)\leq p e^{\int_0^\delta G(r)dr}-(2^*-1).
\end{equation*}
Using that $\lim_{\delta\to 0}\int_0^\delta G(r)dr=0$ and $p<2^*-1$, we are allowed to choose $\delta$ small enough so that
\[p e^{\int_0^\delta G(r)dr}< 2^*-1.\]
Hence, \eqref{increasingineq} holds for $0<s\leq\delta$.

We assume now that $s>\delta$. In this case,
\[\psi(s)=AL+\delta^{\tau}L\frac{\left(\frac{s}{L}\right)^{1-\tau}-\delta^{1-\tau}}{1-\tau},\]
where $\tau=\delta g(x_0,\delta)$ and 
\[A=\frac{\psi(\delta L)}{L}=\int_0^\delta e^{\int_t^\delta g(x_0,r)dr}dt.\]
Recall that $\tau\leq\sigma<\frac{2^*-1-p}{2^*-2}$. Then, 
\[\frac{p-\tau}{1-\tau}\leq\frac{p-\sigma}{1-\sigma}<2^*-1.\]
Let us denote $\gamma=2^*-1-\frac{p-\sigma}{1-\sigma}$. We deduce that
\begin{align*}
(p&-s\bar{g}(x_0,s))\frac{\psi(sL)}{L}- (2^*-1)s\psi'(sL)
\\
&=(p-\tau)\left(A+\delta^{\tau}\frac{s^{1-\tau}-\delta^{1-\tau}}{1-\tau}\right)-(2^*-1)\delta^\tau s^{1-\tau}
\\
&\leq (2^*-1)\left((1-\tau)A+\delta^\tau(s^{1-\tau}-\delta^{1-\tau})-\delta^\tau s^{1-\tau}\right)-\gamma\left(A+\delta^{\tau}\frac{s^{1-\tau}-\delta^{1-\tau}}{1-\tau}\right)
\\
&\leq (2^*-1)(A-\delta)-\gamma A=A\left((2^*-1)\left(1-\frac{\delta}{A}\right)-\gamma\right)
\\
&\leq A\left((2^*-1)\left[1-\delta\left(\int_0^\delta e^{\int_t^\delta G(r)dr}dt\right)^{-1}\right]-\gamma\right).
\end{align*}
It is clear that
\[\lim_{\delta\to 0}\frac{\int_0^\delta e^{\int_t^\delta G(r)dr}dt}{\delta}=\lim_{\delta\to 0}\frac{e^{\int_0^\delta G(r)dr}\int_0^\delta e^{-\int_0^t G(r)dr}dt}{\delta}=\lim_{\delta\to 0}e^{-\int_0^\delta G(r)dr}=1.\]
Then, we can take $\delta$ small enough so that 
\[(2^*-1)\left[1-\delta\left(\int_0^\delta e^{\int_t^\delta G(r)dr}dt\right)^{-1}\right]< \gamma.\]
In conclusion, \eqref{increasingineq} holds for $s>\delta$, and thus, for all $s>0$. 

In sum, it follows that the hypotheses of Proposition~\ref{gidasspruck5} are fulfilled for problem $(\bar{P}_\lambda)$ so our claim was true. The proof is now finished.
\end{proof}

We conclude the section by proving Theorem~\ref{modelthm1}, Theorem~\ref{modelthm2} and Theorem~\ref{modelthm3} as consequences of the general results in Section~\ref{mainresults}.

\begin{proof}[Proof of Theorem~\ref{modelthm1}]
The first item of the theorem is a direct consequence of Theorem~\ref{mainthm1}, while the second one follows from Theorem~\ref{nonexistence}.
\end{proof}

\begin{proof}[Proof of Theorem~\ref{modelthm2}]
Let us denote $g(x,s)=\frac{\mu(x)}{(s+\delta)^\gamma}$. In order to prove the first item, we assume that $\delta>0$ and that $\mu$ satisfies \eqref{musmallcond}. It is easy to see that 
\[\inf_{(x,s)\in\Omega\times [0,+\infty)}(s+\delta)g(x,s)=-\frac{\|\mu^-\|_{L^\infty(\Omega)}}{\delta^{\gamma-1}},\quad \sup_{(x,s)\in\Omega\times [0,+\infty)}(s+\delta)g(x,s)=\frac{\|\mu^+\|_{L^\infty(\Omega)}}{\delta^{\gamma-1}}.\]
Therefore, $\tau\leq (s+\delta)g(x,s)\leq\sigma$ for $\tau=-\frac{\|\mu^-\|_{L^\infty(\Omega)}}{\delta^{\gamma-1}}$ and $\sigma=\frac{\|\mu^+\|_{L^\infty(\Omega)}}{\delta^{\gamma-1}}$. Hence, \eqref{musmallcond} implies that $g$ satisfies \eqref{sgsmallcond}. On the other hand, it is clear that $\|(s+\delta)g(\cdot,s)\|_{L^\infty(\Omega)}\to 0$ as $s\to +\infty$. Thus, $g$ also satisfies \eqref{unifconv}. In conclusion, Theorem~\ref{mainthm1} implies that there exists a finite energy solution to \eqref{model1} for any $\lambda>0$.

On the contrary, let us assume now that $\delta=0$ and $\mu(x)\geq\tau>0$ for a.e. $x\in\Omega\setminus\omega$. Then, \eqref{gtoolarge} holds. Thus, Theorem~\ref{nonexistence} implies that problem \eqref{model1} admits no solution for any $\lambda>0$.
\end{proof}

\begin{proof}[Proof of Theorem~\ref{modelthm3}]
Let us denote $g(x,s)=\frac{\mu(x)}{(s+\delta)^\gamma}$. We know that, for every $\varepsilon>0$, there exists $s_0>0$ such that
\begin{equation}\label{ineqvarepsilon}
s|g(x,s)|\leq\|\mu\|_{L^\infty(\Omega)}\frac{s}{(s+\delta)^\gamma}<\varepsilon\quad\forall s\in(0,s_0),\,\,\text{a.e. }x\in\Omega.
\end{equation}
Therefore, \eqref{sguniflimitzero} holds (in particular, $\lim_{s\to 0}sg(x,s)$ obviously exists for a.e. $x\in\Omega$). If we assume that $0\lneq \mu\in C(\overline{\Omega})$ and $p<2^*-1$, then the hypotheses of Theorem~\ref{mainthm5} are fulfilled, so we conclude.

Let us assume on the contrary that $p<\frac{N+1}{N-1}$. Then, in \eqref{ineqvarepsilon} we may take $\sigma=-\tau=\varepsilon\in\left(0,\min\left\{\frac{1}{3},\frac{N+1-(N-1)p}{2}\right\}\right)$ in such a way that $g$ satisfies \eqref{sgsmallcondtilde}. Thus, the result follows after applying Theorem~\ref{mainthm4}.

The proof is finished.
\end{proof}

\section{Appendix}\label{appendix} 

\subsection{Technical lemma}

In this subsection we prove a technical result that is required by the blow-up method.

\begin{lemma}\label{straighteninglemma}
Let $\Omega\subset\mathbb{R}^N (N\geq 3)$ be a bounded domain with boundary of class $\mathcal{C}^2$, $\lambda>0$, $g:\Omega\times (0,+\infty)\to\mathbb{R}$ be a Carath\'eodory function, and $f:[0,+\infty)\to\mathbb{R}$ be a continuous function. Then, for every $x_0\in\partial\Omega$, there exist $U\subset\mathbb{R}^N$ a neighborhood of $x_0$ and an injective and $C^2$ map $y:U\to \mathbb{R}^N$, with $C^2$ inverse, such that $V=y(U\cap\Omega)\subset\mathbb{R}^N_+$, $\Gamma=y(U\cap\partial\Omega)\subset\partial V\cap \partial\mathbb{R}^N_+$ and, if $u$ is a solution to \eqref{lambdaproblem},
then the function $v=u\circ y^{-1}:V\to (0,+\infty)$ is a solution to
\begin{equation}\label{halfspaceeq}
\begin{cases}
-\div(M(y)\nabla v)+b(y)\nabla v+j(y,v)M(y)\nabla v\nabla v=f(v),\quad &y\in V,
\\
v>0,\quad &y\in V,
\\
v=0,\quad &y\in \Gamma,
\end{cases}
\end{equation}
where $M\in C^1(\overline{V})^{N\times N}$ is uniformly elliptic, $b\in C(\overline{V})^N$ and $j(\cdot,\cdot)=g(y^{-1}(\cdot),\cdot)$, being $M$ and $b$ independent of $u$.
\end{lemma}

\begin{proof}

Since $\partial\Omega$ is of class $\mathcal{C}^2$, there exist $U\subset\mathbb{R}^N$ a neighborhood of $x_0$ and a $C^2$ function $\psi:U'\to\mathbb{R}$, where $U'=\{x'\in\mathbb{R}^{N-1}:\exists x_N\in\mathbb{R}, (x',x_N)\in U\}$, such that 
\begin{align*}
\psi(x')<x_N\quad &\forall(x',x_N)\in U\cap\Omega,
\\
\psi(x')=x_N\quad &\forall(x',x_N)\in U\cap\partial\Omega.
\end{align*}
Let us define the change of variables $y:U\to\mathbb{R}^N$ by
\[y(x)=(x',x_N-\psi(x'))\quad\forall x\in U.\]
It is clear that
\begin{align*}
y(x)\in\mathbb{R}^N_+\quad &\forall(x',x_N)\in U\cap\Omega,
\\
y(x)\in\partial\mathbb{R}^N_+\quad &\forall(x',x_N)\in U\cap\partial\Omega.
\end{align*}
This proves that $V\subset\mathbb{R}^N_+$ and $\Gamma\subset\partial\mathbb{R}^N_+$. Moreover, it is a simple exercise to prove that $\Gamma\subset\partial V$. It is also easy to see that the function $y^{-1}:y(U)\to U$ given by
\[y^{-1}(z)=(z',z_N+\psi(z'))\quad\forall z\in y(U)\]
is the inverse function of $y$. Note that $y^{-1}$ is well defined since $z'\in U'$ for every $z\in y(U)$.

Let us now define $v:V\cup\Gamma\to\mathbb{R}$ by
\[v(z)=u(y^{-1}(z))\quad\forall z\in V\cup\Gamma.\]
Observe that $v=0$ on $\Gamma$ and $u(x)=v(y(x))$ for all $x\in U\cap\Omega$. We will show next that $v$ satisfies an equation in $V$. 

Now we compute the derivatives that we need. We emphasize that such derivatives can be understood in a pointwise sense due to Remark~\ref{strongsolremark} and to the $C^2$ regularity of $\psi$. We stress also that, as $\psi$ does not depend on $x_N$, it will be understood that $\frac{\partial\psi}{\partial x_N}(x)=0$ for $x\in U$.

\begin{enumerate}[label=\alph*)]
\item $\displaystyle Dy(x)=\left(
\begin{array}{cccc|c}
1&0&\cdots&0&0
\\
0&1& &\vdots&\vdots
\\
\vdots& &\ddots&0&\vdots
\\
0&\cdots&0&1&0
\\
\hline
-\frac{\partial\psi}{\partial x_1}(x')&\cdots&\cdots&-\frac{\partial\psi}{\partial x_{N-1}}(x')&1
\end{array}\right),$
\\\vspace{0.5 cm}
\item $\begin{cases}\displaystyle
\frac{\partial u}{\partial x_i}(x)=\frac{\partial v}{\partial z_i}(y(x))-\frac{\partial v}{\partial z_N}(y(x))\frac{\partial\psi}{\partial x_i}(x'),\quad i=1,\cdots, N-1,
\\\\
\displaystyle
\frac{\partial u}{\partial x_N}(x)=\frac{\partial v}{\partial z_N}(y(x)),
\end{cases}$
\\\vspace{0.5 cm}
\item $|\nabla u(x)|^2=|\nabla v(y(x))|^2+\left(\frac{\partial v}{\partial z_N}(y(x))\right)^2|\nabla\psi(x')|^2-2\frac{\partial v}{\partial z_N}(y(x))\nabla v(y(x))\nabla\psi(x'),$
\\\vspace{0.5 cm}
\item \begin{align*} \frac{\partial^2 u}{\partial x_i^2}(x) &= \frac{\partial^2 v}{\partial z_1\partial z_i}(y(x))\frac{\partial y_1}{\partial x_i}(x)+\cdots+\frac{\partial^2 v}{\partial z_N\partial z_i}(y(x))\frac{\partial y_N}{\partial x_i}(x)-\frac{\partial v}{\partial z_N}(y(x))\frac{\partial^2\psi}{\partial x_i^2}(x')
\\
&-\left[\frac{\partial^2 v}{\partial z_1\partial z_N}(y(x))\frac{\partial y_1}{\partial x_i}(x)+\cdots +\frac{\partial^2 v}{\partial z_N^2}(y(x))\frac{\partial y_N}{\partial x_i}(x)\right]\frac{\partial\psi}{\partial x_i}(x')
\\
&=\frac{\partial^2 v}{\partial z_i^2}(y(x))-2\frac{\partial^2 v}{\partial z_i\partial z_N}(y(x))\frac{\partial\psi}{\partial x_i}(x')+\frac{\partial^2 v}{\partial z_N^2}(y(x))\left(\frac{\partial\psi}{\partial x_i}(x')\right)^2
\\
&-\frac{\partial v}{\partial z_N}(y(x))\frac{\partial^2\psi}{\partial x_i^2}(x'),\quad i=1,\cdots, N-1,
\end{align*}
\item $\displaystyle\Delta u=\Delta v-2\nabla\frac{\partial v}{\partial z_N}\nabla\psi+\frac{\partial^2 v}{\partial z_N^2}|\nabla\psi|^2-\frac{\partial v}{\partial z_N}\Delta\psi.$
\end{enumerate}


Let us denote $j(z,s)=g(y^{-1}(z),s)$ for a.e. $z\in V$ and for all $s>0$. Thus, $v=v(y)$ (from this point $y$ will simply denote variable in $V$) satisfies the equation
\begin{align}\label{newequation}
\nonumber -\Delta v &-\frac{\partial^2 v}{\partial y_N^2}|\nabla\psi|^2+2\nabla\frac{\partial v}{\partial y_N}\nabla\psi+\frac{\partial v}{\partial y_N}\Delta\psi
\\
&=f(v)-j(y,v)\left(|\nabla v|^2+\left(\frac{\partial v}{\partial y_N}\right)^2|\nabla\psi|^2-2\frac{\partial v}{\partial y_N}\nabla v\nabla\psi\right),\quad y\in V.
\end{align}

Let us define the matrix
\[M(y)=\left(
\begin{array}{cccc|c}
1&0&\cdots&0&-2\frac{\partial\psi}{\partial x_1}(y')
\\
0&1& &\vdots&\vdots
\\
\vdots& &\ddots&0&\vdots
\\
0&\cdots&0&1&-2\frac{\partial\psi}{\partial x_{N-1}}(y')
\\
\hline
0&\cdots&\cdots&0&1+|\nabla\psi(y')|^2
\end{array}\right),\quad y\in V,\]
and also the vector $b(y)=(0,\cdots,0,-\Delta\psi(y')),\, y\in V$. Then, one can check that $v$ is a solution to \eqref{halfspaceeq}.

Moreover, let $a>1$ be such that $(a-1)\|\nabla\psi\|_{L^\infty(V)}^2<1$. Then, by Cauchy-Schwarz's and Young's inequalities, the following holds for every $\xi\in\mathbb{R}^N$:
\begin{align*}
M(y)\xi\xi&=|\xi|^2-2\xi_N\nabla\psi\xi'+\xi_N^2|\nabla\psi|^2
\\
&\geq |\xi|^2-(a-1)\xi_N^2|\nabla\psi|^2-\frac{1}{a}|\xi'|^2
\\
&=\left(1-\frac{1}{a}\right)|\xi'|^2+(1-(a-1)|\nabla\psi|^2)\xi_N^2
\\
&\geq \left(1-\frac{1}{a}\right)|\xi'|^2+(1-(a-1)\|\nabla\psi\|_{L^\infty(V)}^2)\xi_N^2
\\
&\geq\min\left\{1-\frac{1}{a},1-(a-1)\|\nabla\psi\|_{L^\infty(V)}^2\right\}|\xi|^2.
\end{align*}
Then, $M$ is uniformly elliptic. The proof is finished.
\end{proof}


\subsection{Local H\"older estimate and global continuity}
\text{ }

This subsection is devoted to proving a local H\"older condition on the solutions to the problems that concern this manuscript. In particular, the solutions are continuous in the interior of $\Omega$. We will show at the end of the section that the solutions are, in fact, continuous up to the boundary.

\begin{lemma}\label{reglemma}
Let $g:\Omega\times (0,+\infty)\to\mathbb{R}$ be a Carath\'eodory function satisfying \eqref{ggeneral} and let $h\in L^q(\Omega)$ for some $q>\frac{N}{2}$. Then, for every $M>c>0$ and every $\omega\subset\subset\Omega$, there exist $\alpha\in (0,1)$, $L>0$ such that every solution $u\in H^1_{\mbox{\tiny loc}}(\Omega)\cap L^\infty(\Omega)$ to
\begin{equation}\label{regholdereq}
\begin{cases}
-\Delta u+g(x,u)|\nabla u|^2=h(x), &x\in\Omega,
\\
c\leq u\leq M, &x\in\Omega,
\end{cases}
\end{equation}
satisfies 
\begin{equation}\label{localholdercondition} \|u\|_{C^{0,\alpha}(\overline{\omega})}\leq L.
\end{equation}
\end{lemma}

\begin{proof}
Let us consider $m_0=\max_{s\in[c,M]}g_0(s)$, where $g_0:(0,+\infty)\to [0,+\infty)$ is the continuous function given by \eqref{ggeneral}, and let us take $\delta\in\left(0,\frac{1}{2m_0}\right)$. With the notation of \cite[Chapter~2, Section~6, p.~81]{LU}, we will show that $u$ belongs to the class $\mathcal{B}_2\left(\Omega,M,\gamma,\delta,\frac{1}{2q}\right)$ for some $\gamma>0$ to be determined later. This implies \eqref{localholdercondition} by virtue of \cite[Theorem 6.1, p.90]{LU}.

Indeed, let us fix $\rho>0$ and let $B_\rho$ be an open ball of radius $\rho$ such that $B_\rho\subset\Omega$. Let us fix also $\sigma\in (0,1)$, and consider a function $\zeta\in C_c^\infty(B_\rho)$ satisfying that $0\leq\zeta\leq 1$ in $B_\rho$, $\zeta\equiv 1$ in the concentric open ball $B_{\rho-\sigma\rho}$, and $|\nabla\zeta|<\frac{b}{\sigma\rho}$ in $B_\rho$ for some constant $b>0$ independent of $\rho,\sigma$. 

Let us denote $k_\rho=\sup_{x\in B_\rho}u(x)-\delta$, and take $k\geq k_\rho$. We also consider the set 
\[A_{k,\rho}=\{x\in B_\rho: u(x)\geq k\}.\]
Clearly, $(u-k)^+\zeta^2\in H_0^1(\Omega)\cap L^\infty(\Omega)$ and has compact support in $\Omega$, so it may be taken as test function in \eqref{regholdereq}. Thus we obtain
\[2\int_{A_{k,\rho}}\zeta(u-k)\nabla\zeta\nabla u+\int_{A_{k,\rho}}|\nabla u|^2\zeta^2+\int_{A_{k,\rho}}g(x,u)|\nabla u|^2(u-k)\zeta^2=\int_{A_{k,\rho}}h(x)(u-k)\zeta^2.\]
On the one hand, by \eqref{ggeneral} and by the definitions of $k_\rho, m_0$ we deduce that
\[-\int_{A_{k,\rho}}g(x,u)|\nabla u|^2(u-k)\zeta^2\leq \delta m_0\int_{A_{k,\rho}}|\nabla u|^2\zeta^2.\]
On the other hand, Young's inequality yields
\[-2\int_{A_{k,\rho}}\zeta(u-k)\nabla\zeta\nabla u\leq 2\int_{A_{k,\rho}}(u-k)^2|\nabla\zeta|^2+\frac{1}{2}\int_{A_{k,\rho}}|\nabla u|^2\zeta^2.\]
Therefore, we derive
\[\left(\frac{1}{2}-\delta m_0\right)\int_{A_{k,\rho}}|\nabla u|^2\zeta^2\leq2\int_{A_{k,\rho}}(u-k)^2|\nabla\zeta|^2+\int_{A_{k,\rho}}h(x)(u-k)\zeta^2.\]
Next, it follows from H\"older's inequality that
\[\int_{A_{k,\rho}}h(x)(u-k)\zeta^2\leq\delta\|h\|_{L^q(\Omega)}|A_{k,\rho}|^\frac{1}{q'}\]
and
\begin{align*}
2\int_{A_{k,\rho}}(u-k)^2\zeta^2 &\leq \frac{2b}{\rho^2\sigma^2}\int_{A_{k,\rho}}(u-k)^2\leq \frac{2b}{\rho^2\sigma^2}\|(u-k)^2\|_{L^\infty(A_{k,\rho})}|A_{k,\rho}|
\\
&\leq \frac{2b C(N)}{\rho^{2-\frac{N}{q}}\sigma^2}\|(u-k)^2\|_{L^\infty(A_{k,\rho})}|A_{k,\rho}|^\frac{1}{q'},
\end{align*}
where $|B_\rho|=C(N)^q\rho^N$. In sum,
\begin{equation}\label{LUclassineq}
\int_{A_{k,\rho-\sigma\rho}}|\nabla u|^2\leq \gamma\left(\frac{1}{\rho^{2-\frac{N}{q}}\sigma^2}\|(u-k)^2\|_{L^\infty(A_{k,\rho})}+1\right)|A_{k,\rho}|^\frac{1}{q'}, 
\end{equation}
where 
\[\gamma=\frac{\max\left\{2bC(N),\delta\|h\|_{L^q(\Omega)}\right\}}{\frac{1}{2}-\delta m_0}.\]
Furthermore, notice that $v=-u$ satisfies
\[-\Delta v-g(x,-v)|\nabla v|^2=-h(x),\quad x\in\Omega.\]
Therefore, it is straightforward to check that the inequality \eqref{LUclassineq} also holds by substituting $u$ with $-u$. In conclusion, $u\in \mathcal{B}_2\left(\Omega,M,\gamma,\delta,\frac{1}{2q}\right)$ and the proof is finished.
\end{proof}

\begin{remark}\label{strongsolremark}
Let $g:\Omega\times (0,+\infty)\to\mathbb{R}$ be a Carath\'eodory function satisfying \eqref{ggeneral}, $h\in L^q(\Omega)$ for some $q>\frac{N}{2}$ and $u\in H^1_{\mbox{\tiny loc}}(\Omega)\cap L^\infty(\Omega)$. Using the local H\"older regularity given by Lemma~\ref{reglemma} one can prove that, if $u\in H^1_{\mbox{\tiny loc}}(\Omega)\cap L^\infty(\Omega)$ is a solution to \eqref{holdereq} that is locally bounded away from zero, then $u\in W^{1,2q}_{\mbox{\tiny loc}}(\Omega)$. The proof is based on a standard bootstrap argument, see \cite[Appendix]{CLLM} for further details. Moreover, combining this local summability of the gradients with the classical Calderon-Zygmund regularity theory one can easily prove that, if $h\in L^\infty(\Omega)$, then $u\in W^{2,q}_{\mbox{\tiny loc}}(\Omega)$ for every $q<\infty$. 
\end{remark}

We will show next that every solution to \eqref{hproblem} is continuous up to $\partial\Omega$. 

\begin{lemma}\label{continuouslemma}
Let $g:\Omega\times (0,+\infty)\to[0,+\infty)$ be a Carath\'edory function satisfying \eqref{ggeneral} and let $h\in L^q(\Omega)$ for some $q>\frac{N}{2}$. Then, every solution to \eqref{hproblem} belongs to $C(\overline{\Omega})$.
\end{lemma}

\begin{proof}
Let $u\in H^1_{\mbox{\tiny loc}}(\Omega)\cap L^\infty(\Omega)$ be a solution to \eqref{hproblem}. Clearly, Lemma~\ref{reglemma}  implies that $u\in C(\Omega)$. It is left to prove the continuity on the boundary. In order to do so, recall that there exists $\gamma>0$ such that $u^\gamma\in H_0^1(\Omega)$. We may assume without loss of generality that $\gamma>1$. Let us denote $v=u^\gamma$. Using that $g\geq 0$, we deduce that
\[-\Delta v=-\gamma(\gamma-1)u^{\gamma-2}|\nabla u|^2+\gamma u^{\gamma-1}(-\Delta u)\leq \gamma u^{\gamma-1}h(x)\leq C|h(x)|,\quad x\in\Omega,\]
for some constant $C>0$. On the other hand,
let us consider the unique solution $w\in H_0^1(\Omega)\cap C(\overline{\Omega})$ to the following problem:
\begin{equation*}
\begin{cases}
-\Delta w=|h(x)|, &x\in\Omega,
\\
w=0, &x\in\partial\Omega.
\end{cases}
\end{equation*}
By standard comparison, $0<v\leq w$ in $\Omega$. Therefore, $\lim_{x\to y}v(x)=0$ for every $y\in\partial\Omega$ and, in particular, $u\in C(\overline{\Omega})$.
\end{proof}


\subsection{Global H\"older estimates}
\text{ }

In the previous subsection we have shown that the solutions to the problems in this paper are locally H\"older continuous. In the proof, we have strongly used the fact that the solutions are locally bounded away from zero in order to avoid the possible singularity of $g(x,u(x))$ as $x$ approaches $\partial\Omega$. In the present subsection we aim to prove a \emph{global} H\"older estimate, i.e. the singularity cannot be avoided since we need a H\"older estimate to hold even near $\partial\Omega$. Therefore, we perform a different proof that requires $g$ to satisfy the stronger condition \eqref{sgsmallcond}.

\begin{lemma}\label{holderestimatelemma}
Let $\Omega\subset\mathbb{R}^N$ be a smooth bounded domain. Let $\{M_n\}\subset L^\infty(\Omega)^{N\times N}$ and $\{b_n\}\subset  L^\infty(\Omega)^N$ be bounded sequences in the norm of their respective spaces. Assume that there exists $\eta>0$ such that
\[M_n(y)\xi\xi\geq\eta|\xi|^2\quad\text{a.e. }y\in\Omega,\,\,\forall\xi\in\mathbb{R}^N,\,\,\forall n.\]
For some $p>1$, $\delta,\tau\geq 0$, $\sigma\in(0,1)$ independent of $n$, let $f:[0,+\infty)\to [0,+\infty)$ be a continuous function satisfying \eqref{boundsf} and, for all $n$, let $g_n:\Omega\times (0,+\infty)\to [0,+\infty)$ be a Carath\'edory function satisfying \eqref{sgsmallcond}. Let $\{v_n\}\subset H^1(\Omega)\cap C(\overline{\Omega})$ be such that $0<v_n\leq C$ in $\Omega$ for all $n$ and for some $C>0$. For some sequences $\{L_n\}\subset (0,+\infty)$, $\{\epsilon_n\}\subset [0,1]$, assume that $v_n$ satisfies
\[-\div(M_n(y)\nabla v_n)+b_n(y)\nabla v_n+\frac{1}{L_n}g_n\left(y,\frac{v_n}{L_n}\right)M_n(y)\nabla v_n\nabla v_n=L_n^p f\left(\frac{v_n}{L_n}\right)+\epsilon_n v_n^\sigma,\quad y\in\Omega,\,\,\forall n.\]
If $\delta=0$, assume in addition that there exists $\gamma\in\left(\frac{1-\tau}{2},1-\sigma\right]$ such that $v_n^\gamma\in H^1(\Omega)$ for all $n$. For some (possibly empty) set $\Gamma\subset\partial\Omega$ that is open in the topology of $\partial\Omega$ and connected, let us suppose also that
\[v_n=0,\quad y\in \Gamma.\]
Furthermore, if $\{L_n\}$ diverges, assume additionally that $\delta=0$ and $\epsilon_n=0$ for all $n$. Then, for every smooth bounded domain $\omega$ satisfying that $\overline{\omega}\subseteq\Omega\cup\Gamma$, there exist $\alpha\in (0,1)$, $C>0$ such that, taking a (not relabeled) subsequence if necessary,
\[\|v_n\|_{C^{0,\alpha}(\overline{\omega})}\leq C\quad\forall n.\]
\end{lemma}

\begin{proof}
Let us denote $\delta_n=L_n\delta$ and consider the function $u_n=(v_n+\delta_n)^\gamma-\delta_n^\gamma$ for $\gamma\in\left(\frac{1-\tau}{2},1-\sigma\right]$; if $\delta=0$, then $\gamma$ must be chosen so that $u_n\in H^1(\Omega)$ for all $n$. It is easy to see that $u_n\in C(\overline{\Omega})$, $0<u_n\leq C$ in $\Omega$, $u_n=0$ on $\Gamma$, and it satisfies
\begin{equation}\label{holdereq}
-\div(M_n(y)\nabla u_n)+b_n(y)\nabla u_n=\tilde{g}_n(y,u_n)M_n(y)\nabla u_n\nabla u_n+\tilde{f}_n(u_n),\quad y\in\Omega,
\end{equation}
where
\begin{align*}
\tilde{g}_n(y,s)&=\frac{1-\gamma-\frac{(s+\delta_n^\gamma)^\frac{1}{\gamma}}{L_n}g_n\left(y,\frac{(s+\delta_n^\gamma)^\frac{1}{\gamma}}{L_n}-\delta\right)}{\gamma (s+\delta_n^\gamma)}\quad\text{a.e. }y\in\Omega,\,\,\forall s>0,\,\,\forall n,
\\
\tilde{f}_n(s)&=\gamma (s+\delta_n^\gamma)^\frac{\gamma-1}{\gamma}\left(L_n^p f\left(\frac{(s+\delta_n^\gamma)^\frac{1}{\gamma}}{L_n}-\delta\right)+\epsilon_n \left((s+\delta_n^\gamma)^\frac{1}{\gamma}-\delta_n\right)^\sigma\right)\quad\text{a.e. }y\in\Omega,\,\,\forall n.
\end{align*}

We claim that $\{u_n\tilde{f}_n(u_n)\}$ admits a subsequence bounded in $L^\infty(\Omega)$ and, furthermore,
\begin{equation}\label{sfcond}
0\leq s \tilde{g}_n(y,s)\leq c\quad\text{a.e. }y\in\Omega,\,\,\forall s>0,\,\,\forall n,
\end{equation}
for some $c\in (0,1)$. In order to prove the claim, we first observe that \eqref{boundsf} implies
\[\tilde{f}_n(s)\leq a(s+\delta_n^\gamma)^\frac{p-1+\gamma}{\gamma}+\epsilon_n (s+\delta_n^\gamma)^\frac{\sigma-1+\gamma}{\gamma}.\]
On the one hand, let us assume that $\{L_n\}$ is not divergent. Then, it admits a (not relabeled) bounded subsequence. Hence, taking into account that $p>1$, it clearly holds that \[a u_n(u_n+\delta_n^\gamma)^\frac{p-1+\gamma}{\gamma}\leq C\quad\forall n.\] Besides, since $\gamma\geq\frac{1-\sigma}{2}$, then $\sigma-1+2\gamma\geq 0$, so
\[\gamma\epsilon_n u_n(u_n+\delta_n^\gamma)^\frac{\sigma-1+\gamma}{\gamma}\leq (u_n+\delta_n^\gamma)^\frac{\sigma-1+2\gamma}{\gamma}\leq C\quad\forall n.\]
Thus, $\{u_n\tilde{f}_n(u_n)\}$ is bounded in $L^\infty(\Omega)$. On the other hand, assume that $\{L_n\}$ diverges. Then, $\delta_n=\epsilon_n=0$ for all $n$, so $\tilde{f}_n(u_n)\leq a u_n^\frac{p-1+\gamma}{\gamma}\leq C$ for all $n$ and, again, $\{u_n\tilde{f}_n(u_n)\}$ is bounded in $L^\infty(\Omega)$.

Let us now verify that \eqref{sfcond} holds. Indeed, from \eqref{sgsmallcond} and $\gamma\in\left(\frac{1-\tau}{2},1-\sigma\right]$, it follows
\[0\leq\frac{1-\gamma-\sigma}{\gamma}\leq (s+\delta_n^\gamma)\tilde{g}_n(y,s)\leq \frac{1-\gamma-\tau}{\gamma}<1.\] Thus, our claim holds.

In sum, we may apply the arguments in \cite[Appendix]{CLLM} without relevant changes to prove that $\|u_n\|_{C^{0,\alpha}(\overline{\omega})}\leq C$ for some $C>0, \alpha\in (0,1)$. In particular, using that the function $s\mapsto s^\frac{1}{\gamma}$ is locally Lipschitz for $s\geq 0$, we have that
\begin{align*}
|v_n(x)-v_n(y)|&=|v_n(x)+\delta_n-v_n(y)-\delta_n|
\\
&\leq C|(v_n(x)+\delta_n)^\gamma-(v_n(y)+\delta_n)^\gamma|
\\
&=C|u_n(x)-u_n(y)|\leq C|x-y|^\alpha\quad\forall x,y\in\omega.
\end{align*}
In conclusion, $\|v_n\|_{C^{0,\alpha}(\overline{\omega})}\leq C$, as we wanted to prove.
\end{proof}

\end{document}